\documentclass[aoas,preprint]{imsart}

\RequirePackage[OT1]{fontenc}
\RequirePackage{amsthm,amsmath}
\RequirePackage[numbers]{natbib}
\RequirePackage[colorlinks,citecolor=blue,urlcolor=blue]{hyperref}

\usepackage[all,cmtip]{xy}
\usepackage{xr}
\usepackage{amsmath}
\usepackage{amsthm}
\usepackage{amsfonts}
\usepackage{amssymb}
\usepackage{graphicx}
\usepackage{subfigure}
\usepackage{multirow}
\usepackage{color}
\usepackage{blkarray}
\usepackage{natbib}
\usepackage[dvipsnames]{xcolor}

\newcommand{\bm}[1]{\boldsymbol{#1}}

\renewcommand{\d}{\mathrm{d}}

\newcommand{\Y}{\bm{Y}}
\newcommand{\mE}{\mathcal{E}}

\newcommand{\bN}{\mathbb{N}}
\newcommand{\mB}{\mathcal{B}}
\newcommand{\U}{\bm{U}}
\newcommand{\V}{\bm{V}}
\newcommand{\DD}{\bm{D}}
\newcommand{\mX}{\mathcal{X}}
\newcommand{\mT}{\mathcal{T}}
\newcommand{\wt}[1]{\widetilde{#1}}
\def\mS{\mathcal{S}}
\def\mC{\mathcal{C}}

\newcommand{\F}{\mathcal{F}}

\newcommand{\iid}{\stackrel{iid}{\sim}}

\newcommand{\wh}[1]{\smash{\widehat{#1}}}

\def\C {\,|\:}

\def\C {\,|\:}

\def\mF{\mathcal{F}}

\def\b{\bm{\beta}}
\def\Y{\bm{Y}}

\def\q{\bm{q}}
\def\A{\bm{A}}

\def\x{\bm{x}}
\def\y{\bm{y}}

\def\mV{\mathcal{V}}

\def\BOmega{\bm{\Omega}}

\def\b{\bm{\beta}}

\def\bmS{\bm{\mathcal{S}}}

\renewcommand{\d}{\mathrm{d}\,}
\newcommand{\e}{\mathrm{e}}

\renewcommand{\1}{\mathbb{I}}
\newcommand{\N}{\mathbb{N}}
\newcommand{\R}{\mathbb{R}}
\newcommand{\Ha}{\mathcal{H}^\alpha}

\newtheorem{definition}{Definition}[section]
\newtheorem{lemma}{Lemma}[section]

\newtheorem{theorem}{Theorem}[section]

\newtheorem{remark}{Remark}[section]
\newtheorem{corollary}{Corollary}[section]

 \theoremstyle{assumption}

\startlocaldefs
\numberwithin{equation}{section}
\theoremstyle{plain}

\endlocaldefs

\begin{document}

\begin{frontmatter}
\title{Posterior Concentration for Bayesian  Regression Trees and Forests}
\runtitle{Posterior Concentration for Regression Trees}
\begin{aug}

\author{\fnms{Veronika} \snm{Ro\v{c}kov\'{a} and}\thanksref{t2,m2}\ead[label=e1]{first@somewhere.com}},
\author{\fnms{St\'{e}phanie} \snm{van der Pas}\thanksref{t1,m1}\ead[label=e2]{second@somewhere.com}  {$\quad$}\\
{\normalsize \sl Revision 29 May  2019\\
 {\small First Draft: 29 August 2017}}}\\


\thankstext{t2}{
veronika.rockova@chicagobooth.edu; \hspace{12cm}
This This work was supported by the James S. Kemper Foundation Faculty Research Fund at  the University of Chicago Booth School of Business.}
\thankstext{t1}{
svdpas@math.leidenuniv.nl}


\affiliation{University of Chicago \thanksmark{m2}}
\affiliation{Leiden University \thanksmark{m1}}

\address{}

\end{aug}

\begin{abstract}
Since their inception in the 1980's, regression trees have been one of the more widely used non-parametric prediction methods.
Tree-structured methods yield a histogram reconstruction of the regression surface,
where the bins correspond to terminal nodes of recursive partitioning. Trees are powerful, yet  susceptible to over-fitting. 
Strategies against overfitting have traditionally relied on  pruning  greedily grown trees. The Bayesian framework offers an alternative remedy against overfitting through priors. Roughly speaking, a good prior  charges smaller trees where overfitting does not occur.
 While the consistency of random histograms, trees and their ensembles  has been studied quite extensively, the theoretical understanding of the Bayesian counterparts has  been  missing. In this paper, we take a step 
towards understanding why/when do Bayesian trees and forests not overfit. To address this question, we study the speed at which the posterior concentrates around the true smooth regression function.
We propose a {\sl spike-and-tree} variant of the popular Bayesian CART prior and establish new theoretical results showing that  regression trees (and forests) (a) are capable of recovering smooth regression surfaces {(with smoothness not exceeding one)}, achieving optimal rates up to a log factor, (b) can adapt to the unknown level of smoothness and (c) can perform effective dimension reduction when $p>n$. These results  provide a piece of missing theoretical evidence explaining why Bayesian trees (and additive variants thereof) have worked so well in practice.
 
\end{abstract}



\end{frontmatter}

\section{Non-parametric Regression Setup}\label{sec:back}
The remarkable empirical success  of Bayesian tree-based regression \citep{cart1,cart2,bart} has raised considerable interest in understanding why and when these methods produce good results. 
Despite their extensive use in practice,  theoretical justifications have, thus far, been unavailable. To narrow this yawning gap, we consider the fundamental problem of making inference about an unknown regression function.

Our setup consists of  the nonparametric regression model
\begin{equation}\label{npr}
Y_i=f_0(\x_i)+\varepsilon_i,\quad \varepsilon_i\iid\mathcal{N}(0,1),
\end{equation}
where output variables  $\Y^{(n)}=(Y_{1},\dots, Y_{n})'$  are related in a stochastic fashion to a set of  $p$ potential covariates $\x_i=(x_{i1},\dots,x_{ip})',1\leq i\leq n$. We assume that the covariates $\x_i=(x_{i1},\dots,x_{ip})'$  are fixed and have been rescaled so that $\x_i\in[0,1]^p$. The true unknown regression surface $f_0(\x)$ will be assumed to be smooth,  possibly involving only a small fraction $q_0$ of the $p$ potential covariates.

In the absence of a parametric model, a natural strategy to estimate the unknown regression function is by  partitioning the covariate space into cells and then estimating the function locally within each cell from available observations. Such strategies yield  histogram  reconstructions of the regression surface and have been analyzed theoretically by multiple authors \citep{nobel,devroey_gyorfi,stone77}. 
Regression trees \cite{breiman_book} are among the most popular data-dependent histogram methods, where the  partitioning scheme is obtained through nested parallel axis splitting.
Trees are an integral constituent of ensemble methods that aggregate single tree learners into forests to boost prediction \cite{breiman2}. Tree-based regression, either single or ensemble, is arguably one of the most popular machine learning tools today. In particular, Bayesian variants of these methods {(Bayesian CART and BART)} \citep{bart,cart2} have earned a prominent role  as one of the top machine learners. While consistency results for classical trees and random forests have been available \cite{gordon_olshen1,gordon_olshen2,donoho,breiman_book,biau,biau_etal,scornet,wager2015}, theory on the also  very widely used Bayesian counterparts  is non-existent. 
 Our goal in this paper is to provide the first frequentist optimality results for {\sl Bayesian} trees, and their ensembles.

 Most of  the work on Bayesian nonparametric regression has revolved around Gaussian processes \cite{vdv_zanten,bhatta,yang}. { While there are multiple results on recursive partitioning (or histogram)  priors for Bayesian density estimation \cite{castillo_polya,scricciolo,lu, castillo_rous} (or non-linear autoregression \cite{ghosal_etal})},  the literature on Bayesian regression histograms  is far more deserted.  
One fundamental contribution is due to Coram and Lalley \cite{coram}, who  showed consistency of Bayesian binary regression with uniform mixture priors on step functions and with one predictor. 
More recently, van der Pas and Ro\v{c}kov\'{a} \cite{vdp_rockova} considered a similar setup for estimating step mean functions in Gaussian regression, again with a single predictor. 
This paper goes far beyond that framework,  addressing (a) the full-fledged high-dimensional setup with a diverging number of potential covariates, (b) tree ensembles for additive regression.

The purpose of this paper is to study the rate of convergence of posterior distributions induced by step function priors on the regression surface when $p>n$. 
The speed of convergence is measured by  the size of the smallest shrinking ball around $f_0$ that contains most of the posterior probability. In pioneering works,  Ghosal, Ghosh and Van der Vaart \cite{ghosal_etal} and Shen and Wasserman \cite{shen_wasser} obtained rates of convergence for infinite-dimensional parametric models with iid observations in terms of the size of the model (measured by the metric entropy) and concentration rate of the prior around $f_0$.   These results were later extended to   infinite-dimensional models that are {\sl not {iid}} by
Ghosal and Van der Vaart  \cite{ghosal_vdv}. Their general conceptual framework serves as an umbrella for our development.

\subsection{Our Contributions}
We initially assume that $f_0$ is H\"{o}lder continuous {(with smoothness not exceeding one)} and may depend only on a small fraction of $q_0$ predictors. The optimal rate of estimation of a $q_0$-variable function, which is known to be $\alpha$-smooth, is  $n^{-\alpha/(2\alpha+q_0)}$ \cite{stone_82}. 
Our first result shows that, with suitable regularization priors, single Bayesian regression trees achieve this minimax  rate (up to a log factor). 
In other words, the posterior behaves nearly as well as if we knew $\alpha$ and the number of active covariates $q_0$, concentrating at a rate that only depends on the number of active predictors. This is the first optimality result for Bayesian regression trees, demonstrating their adaptability and reluctance to overfit in high-dimensional scenarios with $p>n$. The regularization is achieved through our proposed {\sl spike-and-tree} prior, a new variant of the Bayesian CART prior  for dimension reduction and model-free variable selection.
Going further, we show that Bayesian {\sl additive regression trees} also achieve  (near) minimax-rate optimal performance when approximating a single smooth function. Finally, the tree ensembles are {\sl also}  shown to be certifiably optimal when the true function is an actual sum of smooth functions, again concentrating at a near minimax rate.


\subsection{Notation}
The notation $\lesssim$ will be used to denote  inequality up to a constant, {$a\asymp b$ denotes $a\lesssim b$ and $b\lesssim a$ and $a\vee b$ denotes $\max\{a,b\}$}. The $\varepsilon$-covering number of a set $\Omega$ for a semimetric $d$, denoted by $N(\varepsilon,\Omega,d),$ is the minimal number of $d$-balls of radius $\varepsilon$ needed to cover set $\Omega$. We denote by $\phi(\cdot;\sigma^2)$ the normal density with  zero mean and  variance $\sigma^2$ and by  $P_f^n= \bigotimes P_{f,i}$ the $n$-fold product measure of the $n$ independent observations under \eqref{npr} with a regression function $f$.
By $\mathbb{P}_n^{\x}=\frac{1}{n}\sum_{i=1}^n\delta_{\x_i}$ we denote the empirical distribution of the observed covariates and let $||\cdot||_n$ denote the empirical norm on $L_2(\mathbb{P}_n^{\x})$. With $||\cdot||_2$ we denote the standard Euclidean norm. 
For $\x\in \R^p$  we denote with $\x_\mS$ the subvector of $\x$ indexed by $\mS\subseteq\{1,\dots,p\}$. 
{With $\mathcal{C}^r$ we denote $r$-times continuously differentiable functions on $[0,1]$.}


\subsection{Outline}
We outline our goals and strategy in Section \ref{sec:back}. We then review several useful concepts for analyzing recursive partitioning schemes in Section \ref{sec:partitions}. In Section \ref{sec:singletree}, we state  our first main result on the posterior concentration for Bayesian CART.  In Section \ref{sec:treeensembles}, we develop tools for analyzing Bayesian additive regression trees and show their optimal posterior concentration in non-additive regression. Section \ref{sec:additiveensembles} presents the final development of our theory concerning the recovery of an additive regression function with additive trees. We conclude with a discussion in Section \ref{sec:discussion}. The proofs of our main theorems are presented in Sections \ref{proof:thm:trees}, \ref{sec:proof:additivetrees} and \ref{proof:thm:weak_learner}.

\section{Background}\label{sec:back}
In this section we lay down rudiments of our modus operandi. Our setup comprises  a sequence of statistical experiments with observations $\Y^{(n)}=(Y_1,\dots,Y_n)'$ and models $P_{f}^n$ defined in \eqref{npr}. Each model  $P_{f}^n$ is parametrized by a regression function $f:[0,1]^p\rightarrow \R$ that lives in an infinite-dimensional space $\mathcal{F}$ endowed with a prior distribution. With adequate priors, the posterior can exhibit nice frequentist properties, which get passed onto its location/scale summary measures. One such property is the ability to pile up in shrinking neighborhoods  around the true regression function $f_0$.  The speed at which the shrinking occurs is the posterior concentration rate and assesses the quality of the posterior beyond just the mere fact that  it is consistent. In our setup, we investigate such concentration properties in terms of $||\cdot||_n$ neighborhoods of $f_0$, where  
$$
||f-f_0||_n^2=\frac{1}{n}\sum_{i=1}^n[f(\x_i)-f_0(\x_i)]^2.
$$
The key to our approach will be drawing upon the foundational  posterior concentration theory for non-iid observations, laid down in the seminal paper by Ghosal and Van der Vaart \cite{ghosal_vdv}.  

Our results are obtained under a unifying hat of a general result which requires three conditions to hold.
Namely, suppose that  for a sequence $\varepsilon_n^2\rightarrow 0$ such that $n\,\varepsilon_n^2$ is bounded away from zero
 and sets $\mathcal{F}_n\subset \mathcal{F}$ we have

\vspace{-0.5cm}
\begin{equation}\label{eq:entropy}
\displaystyle \sup_{\varepsilon > \varepsilon_n} \log N\left(\tfrac{\varepsilon}{36}, \{f \in \mathcal{F}_n: \|f - f_0\|_n < \varepsilon\}, \|.\|_n\right) \leq n\varepsilon_n^2
\end{equation}
\vspace{-0.7cm}
\begin{equation}\label{eq:prior}
\displaystyle {\Pi(f \in \mathcal{F}:  \|f - f_0\|_n \leq  \varepsilon_n)}\geq \e^{-d\,n\,\varepsilon_n^2 }
\end{equation} 
\vspace{-0.7cm}
\begin{equation}\label{eq:remain}
\displaystyle \Pi(\mathcal{F} \backslash \mathcal{F}_n) = o(\e^{-(d+2)\,n\varepsilon_n^2})
\end{equation}
 for some $d>2$. Then it follows from Theorem 4 in \cite{ghosal_vdv} that the posterior distribution concentrates at the rate $\varepsilon_n^2$, i.e.
\begin{equation}\label{eq:post_conc}
\Pi\left(f\in\mathcal{F}: ||f_0-f||_n>M_n\,\varepsilon_n\C\bm{Y}^{(n)}\right)\rightarrow 0
\end{equation}
in $P_{f_0}^n$-probability, as $n\rightarrow\infty$, for any $M_n\rightarrow\infty$.

The conditions of  Ghosal and Van der Vaart \cite{ghosal_vdv} provide a very general recipe for showing posterior concentration in infinite-dimensional models. 
Our {goal in this paper is to obtain tailored statements  for Bayesian regression trees and forests.}
The major challenge  will be (a) designing a sequence of approximating spaces (a sieve) $\mathcal{F}_n\subset \mathcal{F}$ and (b) endowing  $\mathcal{F}$ with a prior distribution  such that the three conditions hold simultaneously for $\varepsilon_n$ as small as possible. To this end, we will build on, and  develop, tools for an analysis of recursive partitioning schemes. 

Throughout this work, we assume that the true regression function $f_0$ is 
H\"{o}lder continuous {and the smoothness parameter does not exceed one}, (or an additive composition thereof), in the sense made precise below.

\begin{definition}\label{holder_smooth}
With $\Ha_p$ we denote the space of uniformly $\alpha$-H\"{o}lder continuous functions, i.e.
\begin{align*}
\Ha_p=\left\{f:[0,1]^p\rightarrow \R;\,\,  ||f||_{\Ha}\equiv \sup_{\x,\y\in[0,1]^p}\frac{|f(\x)-f(\y)|}{ ||\x-\y||_2^\alpha}<\infty \right\},
\end{align*}
where $\alpha\in (0,1]$ and where $ ||f||_{\Ha}$ is the H\"{o}lder coefficient.
\end{definition}


{
The assumption $\alpha \leq 1$  is standard in the study of piecewise constant estimators and priors, see for example \cite{ghosal_vdv, scricciolo,  Engel1994}.  The reason for this limitation is that step functions are relatively rough; e.g. the approximation error of histograms for functions that are smoother than Lipschitz is at least of the order $1/K$, where $K$ is the number of  bins. The number of steps required to approximate a smooth function well   is thus too large, creating a costly bias-variance tradeoff.}

In some applications, it is reasonable to expect that the regression function $f_0$ depends only on a small fraction of input covariates.  For a set of indices $\mS\subseteq\{1,\dots,p\}$, we define
\begin{align*}
\mC(\mS)=\left\{f:[0,1]^p\rightarrow \R; f\,\,\text{is constant in directions}\, \{1,\dots,p\}\backslash \mS \right\}.
\end{align*}
 We will consider two estimation regimes:
\begin{itemize}
\item[(R1)] Regime 1: $f_0$ is $\alpha$-H\"{o}lder continuous and depends on an unknown subset $\mS_0$ of  $|S_0|=q_0$ covariates, i.e. {\sl $f_0\in\Ha_p\,\cap\, \mC(\mS_0)$. }
\item[(R2)] Regime 2: $f_0$ is an aggregate of $T_0$ $\alpha^t$-H\"{o}lder continuous functions $f_0^t$, $1\leq t\leq T_0$, each depending on an unknown subset $\mS_0^t$ of  $|S_0^t|=q_0^t$ covariates, i.e. $f_0(\x)=\sum_{t=1}^{T_0}f_0^t(\x)$, where $f_0^t\in\mathcal{H}_p^{\alpha^t}\,\cap\, \mC(\mS_0^t)$. 
\end{itemize}
For an estimation procedure to be successful in Regime 1, it needs to be doubly adaptive (with respect to $\alpha$ and  $q_0$).  We will show that both single trees and tree ensembles adapt accordingly, performing at a near-minimax rate. Regime 2 makes the performance discrepancies more apparent, where the additive structure of $f_0$ is appreciated by tree ensembles, which can achieve a faster convergence rate than single trees. A variant of Regime 2 was  studied by  \cite{yang} who derived minimax rates for estimating additive smooth functions and showed optimal concentration of additive Gaussian processes. 
Here, we  approximate $f_0$ with step functions and their aggregates. We limit considerations to step functions that are  underpinned by recursive partitioning schemes. 




\section{On  Recursive Partitions}\label{sec:partitions}

Tree-based regression  consists of first finding an underlying partitioning  scheme that hierarchically subdivides a dataset  into more homogeneous subsets, and then  learning a piecewise constant function on that partition.
This section serves to review several useful concepts for analyzing  such nested partitioning rules that will be instrumental in our analysis.


\subsection{General Partitions}
Given $K\in\bN$,  we define a $K$-{\sl partition} of $[0,1]^p$ as a sequence  $\bm{\Omega}=\{\Omega_k\}_{k=1}^K$ of $K$ contiguous {\sl rectangles} $\Omega_k\subset [0,1]^p$, where  $\cup_{k=1}^K\Omega_k=[0,1]^p$.   Sufficient conditions for consistency of regression histograms have traditionally revolved around two requirements on the partitioning cells $\Omega_k$ (Section 6.3 in \cite{devroye_book}). 
 The first one  pertains to cell counts, where $\Omega_k$'s should be large enough to capture a sufficient number of points to render local estimation meaningful. 
 The second one pertains to cell sizes, where $\Omega_k$'s should be small enough to detect local changes in the regression surface. We borrow some of these concepts for our theoretical analysis.

For the first requirement, let us formalize the notion of the cell size in terms of the empirical measure induced by observations $\mX=\{\x_1,\dots,\x_n\}$.   For each cell $\Omega_k$, we define by
\begin{equation}\label{muk}
\mu(\Omega_k)=\frac{1}{n}\sum_{i=1}^n\mathbb{I}(\x_i\in\Omega_k)
\end{equation}
 the proportion of observations falling inside $\Omega_k$. Throughout this work, we focus on partitions whose boxes can adaptively stretch and shrink, allowing the splits to arrange themselves in a data-dependent way \cite{stone85}.
The simplest data-adaptive partition is based on statistically equivalent blocks \cite{anderson1966, devroye_book}, where all cells have { approximately the same number} of points, i.e. {$\mu(\Omega_k)\asymp1/K$}.
We deviate from such a strict rule  by allowing for imbalance and define the so called {\sl valid} partitions. 

\begin{definition}(Valid Partitions)\label{ass:balance}
Denote by $\BOmega=\{\Omega_k\}_{k=1}^K$ a partition of $[0,1]^p$. We say that  $\BOmega$ is {\sl valid} if
\begin{equation}\label{eq:balance}
\mu(\Omega_k)\geq \bar{C}^2/n\quad \quad\text{for all}\quad k=1,\dots, K
\end{equation}
for some constant $\bar{C}^2\in\N\backslash\{0\}$.
\end{definition}
Valid partitions have non-empty cells, where the allocation {\sl does not} need to be  balanced. In {\sl balanced partitions}  (introduced in van der Pas and Ro\v{c}kov\'{a} \citep{vdp_rockova}), the cells satisfy
$ \frac{C_{min}^2}{K}\leq \mu(\Omega_k)\leq \frac{C_{max}^2}{K}$ for some $C_{min}<1<C_{max}$.
One prominent example of such balanced partitions is the median tree (or a $k$-$d$ tree)  \cite{bentley1975}, which will be discussed in the next section and will be used as a benchmark  tree  approximation towards establishing  Condition \ref{eq:prior}.

For the second requirement, let us formalize the notion of the cell size in terms of the local  spread of the data. To this end, we introduce the partition {\sl diameter} \cite{verma2009, kpotufe2010}.
 


\begin{definition}(Diameter)
Denote by $\bm{\Omega}=\{\Omega_k\}_{k=1}^K$  a partition of $[0,1]^p$ and by $\mX=\{\x_1,\dots,\x_n\}$ a collection of data points in $[0,1]^p$.  For an index set $\mS\subseteq\{1,\dots, p\}$, we define a {\sl diameter} of $\Omega_k$ w.r.t. $\mS$ as 
$$
\mathrm{diam}\left(\Omega_k;\mS\right)=\max\limits_{\substack{\x,\y\in\Omega_k\cap\mX}}\|\x_\mS- \y_\mS\|_2.
$$
and with
$
\mathrm{diam}\left(\bm{\Omega};\mS\right)=\sqrt{\sum_{k=1}^K\mu(\Omega_k)\mathrm{diam}^2\left(\Omega_k;\mS\right)}
$
we define a {\sl  diameter} of the entire partition $\BOmega$  w.r.t. $\mS$.
\end{definition}

The diameter of $\Omega_k$ corresponds to the largest $||\cdot||_2$ distance between $\mS$-coordinate projections of two points that fall inside $\Omega_k$.  
This is one of the more flexible notions of a cell diameter, which  takes into account the data itself, not just the physical cell size. Traditionally,  bias and convergence rate analysis of  tree-based estimators have been characterized
in terms of the cell diameters. The rate depends on how fast the diameters shrink as we move down the tree: the more rapidly, the better. As will be seen later in Section \ref{sec:step_functions}, controlling the diameter will be essential for obtaining tight bounds on the approximation error.



\subsection{Tree Partitions}
We  are ultimately interested in partitions that can be obtained with {\sl nested} axis-parallel splits. {Such partitions can be represented by a tree diagram, a hierarchical arrangement of nodes. We will focus on binary tree partitions, where each split yields two children nodes. Namely, starting from a parent node $[0,1]^p$, a binary tree partition is grown by successively applying a  splitting rule on a chosen internal node, say {${\Omega}_k^\star\subset[0,1]^p$.} 
The  node ${\Omega}_k^\star$ is split into two cells by a perpendicular bisection  along one of the $p$ coordinates, say $j$, at a chosen observed value $\tau\in\mX=\{\x_1,\dots,\x_n\}$. The newborn cells $\{\x\in{\Omega}_k^\star:x_j\leq \tau\}$ and $\{\x\in{\Omega}_k^\star:x_j>\tau\}$ constitute two rectangular regions of $[0,1]^p$, which can be split further (to become internal nodes) or end up being  terminal nodes . The terminal cells after $K-1$ splits then yield a box-shaped (tree) partition $\{\Omega_k\}_{k=1}^K$.}

Unlike {\sl dyadic} trees, where the threshold $\tau$ is preset at a midpoint of the rectangle ${\Omega}_k$ along the $j^{th}$ direction, we allow for splits
at available observations $\mX$. {Such data-dependent splits are integral to Bayesian CART and BART implementations \citep{cart1,cart2,bart}.} With more opportunities {for each split}, many more tree topologies can be generated. However,  since the tree partitions are arranged in a nested fashion, their combinatorial complexity is not too large (as shown in Lemma \ref{lemma:complexity} below).  

We will denote each tree-structured $K$-partition by $\mT^K=\{\Omega_k\}_{k=1}^K$.
 With $\mV^K_\mS$ we denote a {\sl family of valid tree partitions} of  $[0,1]^p$, obtained by splitting $K-1$ times at observed values in  $\mX$ along each coordinate direction inside $\mS\subseteq\{1,\dots,p\}$ at least once. In particular, each tree $\mT\in\mV^K_\mS$ is valid according to Definition \ref{ass:balance} and uses up all covariates in $\mS$ for splits, where $\mS$ can be regarded as an index set of active predictors. We will refer to the partitioning number $\Delta(\mV^K_\mS)$ (in a similar vein as in \cite{nobel}) as the overall number of distinct partitions  of $\mX$ that can be induced by  members of the partitioning family $\mV^K_\mS$.

{
\begin{lemma}\label{lemma:complexity}
Denote by $\mS\subseteq\{1,\dots,p\}$ an index set of active covariates. 
Let $\mV^K_\mS$ denote the set of valid tree partitions {obtained with $K-1$ splits}. Then
$$
\Delta(\mV^K_\mS) = \frac{ |\mS|^{K-1}n!}{(n-K+1)!}.
$$
\end{lemma}
\vspace{-0.5cm}
\begin{proof}
 This follows from the recursive formula
 $
 \Delta(\mV^K_\mS)=\Delta(\mV^{K-1}_\mS)(n-K+2) |\mS|
 $, where we use the fact that there are $\Delta(\mV^{K-1}_\mS)$ possible trees with $K-1$ cells which have altogether $n-K+2$ potential next splits along one of the $|\mS|$ coordinates. 
\end{proof}}
Lemma \ref{lemma:complexity} will be fundamental for understanding the  combinatorial richness of trees  and  for obtaining bounds on the  covering numbers towards establishing Condition \eqref{eq:entropy}.

\begin{remark}(The $k$-$d$ tree partition)\label{remark:kd}
We now pause to revisit one of the most popular space partitioning structures, the $k$-$d$ {tree partition} \cite{bentley1975}. 
Such a partition $\wh{\mT}$ is constructed by cycling over coordinate directions in $\mS$, where all nodes at the same level are split along the same axis. For a given direction $j\in\mS$, each  internal node, say $\Omega_k^\star$, will be split  at a median of the point set (along the $j^{th}$ axis). This split will pass $\lfloor \mu(\Omega_k^\star) n/2\rfloor$ and $\lceil \mu(\Omega_k^\star) n/2\rceil$ observations onto its two children, thereby roughly halving the number of points. 
After $s$ rounds of splits on each variable, all  $K$ terminal nodes have at least $\lfloor n/K \rfloor$ observations, where $K=2^{s\, |\mS|}$. The $k$-$d$ tree partitions are thus balanced in light of Definition 2.4 of van der Pas and Ro\v{c}kov\'{a} \cite{vdp_rockova}.
{While $k$-$d$ trees are not adaptive to intrinsic dimensionality of data (in comparison with  more flexible partitions such as random projections trees \citep{verma2009})}, the diameters of  $k$-$d$ tree partitions can shrink fast, as long as the number of directions is not too large. In particular, 
$\mathrm{diam}(\wh{\mT};\mS)\leq  \sqrt{|\mS|}/K^{1/(2\,|\mS|)} $ for $K=2^{s\,|\mS|}$ for some $s\geq 1$ according to Proposition 6 of \cite{verma2009}. The $k$-$d$ tree construction will be instrumental in establishing Condition \eqref{eq:prior}. 
 \end{remark}

\subsection{On Tree-structured Step Functions}\label{sec:step_functions}
 The second critical ingredient in building a tree regressor is learning the piecewise function on a given partition.
 In this section, we describe some facts about the approximating properties of  such  tree-structured step functions {(further referred to as trees)}.
 The understanding of how well we can approximate a smooth regression surface  will be vital for establishing Condition
  \eqref{eq:prior}. 
  
For a family of {\sl valid} tree partitions $\mV^K_{\mS}$, we denote by
\begin{align}\label{class:stepf}
\mathcal{F}(\mV^K_{\mS}) &= \left\{f_{\mT,\b}: [0,1]^p \to \R ; f_{\mT,\b}(\x) = \sum_{k=1}^K \beta_k\,\1_{\Omega_k}(\x);\mT\in\mV^K_{\mS},\b\in\R^{K}\right\}
\end{align}
the set of all step functions supported by members of the tree partitioning family $\mV_\mS^K$. Each function is underpinned by a valid tree partition $\mT=\{\Omega_k\}_{k=1}^K\in\mV_\mS^K$ and a vector of step heights $\b=(\beta_1,\dots,\beta_K)'\in\R^K$.


The cell diameters oversee how closely one can approximate $f \in \Ha_{p}\,\cap\,\mC(\mS)$ with  $f_{\mT,\b}\in \mathcal{F}(\mV^K_{\mS})$ {and it is desirable that they decay fast with $K$.
Ideally, the approximation error should be no slower than $q^\gamma/K^{1/q}$ for some $\gamma>0$, where $q=|\mS|$.
To  get an intuitive insight into this requirement, consider the following perfectly regular partition: a $q$-dimensional ``chess-board" that splits $[0,1]^{q}$ into $K=s^q$ cubes of length $1/s=1/K^{1/q}$. The maximal interpoint distance inside each cube will be at most $\sqrt{q}/K^{1/q}$. This partition is, however, not adaptive and thereby less practical. It turns out that, under a mild requirement on the spread of the data points $\mX$,  the fast diameter decay is {\sl also} guaranteed by the data-adaptive $k$-$d$ trees mentioned in Remark \ref{remark:kd}. The ``mild requirement" is formalized in our definition below.}

\begin{definition}\label{def:regular}
Denote by $\wh{\mT}=\{\wh{\Omega}_k\}_{k=1}^K\in\mV_{\mS}^K$ the $k$-$d$ tree where  $\mS\subseteq\{1,\dots,p\}$ and $K=2^{s\,|\mS|}$. We say that a dataset $\mX$ is $(M,\mS)$-regular if  
\begin{equation}\label{eq:design_regular}
\max\limits_{1\leq k\leq K}\mathrm{diam}(\wh{\Omega}_k;\mS)<M\, \sum_{k=1}^K\mu(\wh{\Omega}_k)\mathrm{diam}(\wh{\Omega}_k;\mS)
\end{equation}
for some large enough constant $M>0$ and all  $s\in\N\backslash\{0\}$.
\end{definition}
The definition states  that in a regular dataset,  the maximal diameter  in the $k$-$d$ tree partition should not be much larger than a ``typical" diameter.  This condition ensures that, as more and more data points are collected, the data conform to a structure  that does not permit outliers in active directions $\mS$. 
{For example, a fixed design on a regular grid (including directions $\mS$) will satisfy \eqref{eq:design_regular}. Our notion of regularity goes farther by allowing (a) the predictors to  be correlated, (b)
 the points to scatter unevenly and/or close to a lower-dimensional manifold. The manifold, however,  should be  varying  in active directions $\mS$ and do so sufficiently {\sl smoothly} (or be monotone) so that the cells in the $k$-$d$ tree do not contain isolated clouds of points. For example, data points arising from distributions with atomic marginals (in active directions) would violate regularity.}

{As will be seen in the following lemma, for regular datasets, $k$-$d$ trees have a faster diameter decay (roughly halving the cell diameters after one round of $q=|\mS|$ splits), thereby attaining better approximation error. }
The following lemma is an important element of our proof skeleton, showing that there {\sl exists} a tree (a $k$-$d$ tree) that approximates well. 

\begin{lemma}\label{lemma:approx}
Assume $f \in \Ha_{p}\,\cap\,\mC(\mS)$, where $|\mS|=q$ { and $\alpha \leq 1$}, and that $\mX$ is $(M,\mS)$-regular. Then there exists  a tree-structured step function $f_{\wh{\mT},\,\wh{\b}}\in\mV_{\mS}^K$ for some $K=2^{s\,q}$ with $s\in\N\backslash\{0\}$ such that 
\begin{equation}\label{approx2}
\|f - f_{\wh{\mT},\,\wh{\b}}\|_n \leq   ||f||_{\Ha}C\, M^\alpha\, q/K^{\alpha/q}
\end{equation}
for some $C>0$.
\end{lemma}
\begin{proof}
Supplemental Material (Section 3).
\end{proof}

Controlling the approximation error is only one of the critical aspects in our theoretical study. The second will be monitoring the complexity of our approximating function classes. Finding the right balance between the two will be instrumental for obtaining optimal performance.

Now that we have developed the necessary tools, we can dive into the main results of the paper.

\section{Adaptive Dimension Reduction with Trees}\label{sec:singletree}
In Regime 1, we assume  that the target regression surface $f_0\in\Ha_p\cap\mC(\mS_0)$, although initially conceived as a function of $\x=(x_1,\dots,x_p)'$, in fact depends only on a small fraction of $q_0$ features $\x_{\mS_0}$, where $\mS_0\subset\{1,\dots, p\}$.  If an oracle were able to isolate $\mS_0$, the {$L_2$ minimax rate} would improve from  $n^{-\alpha/(2\alpha+p)}$ to $n^{-\alpha/(2\alpha+q_0)}$ and it would be the fastest  rate possible. 
When no such oracle information is available,  \cite{yang} characterized the minimax rate as follows:  {$\varepsilon_n^2= \lambda^2\left(\sqrt{n}\lambda\right)^{-4\alpha/(2\alpha+q_0)}+\frac{q_0}{n}\log\left(\frac{p}{q_0}\right)$
where $\lambda$ is the isotropic H\"{o}lder norm.} The first term is the classical minimax risk for estimating a $q_0$-dimensional smooth function.  The second term is the penalty incurred by  variable selection uncertainty. 
While the number of features $p$ is not forbidden from growing to infinity much faster than $n$, {we keep in mind that consistent estimation will only be possible  in sparse contexts where $q_0$ is small relative to $p$  and $n$ (in which case the complexity penalty will be dominated by the first term in the minimax rate).}

\subsection{Spike-and-Tree Priors}\label{sec:spike_tree}
{Bayesian regression tree
implementations that do not induce sparsity (when in fact present)
are unfit for inference in  high-dimensional setups.}
In particular, the traditional Bayesian CART prior \cite{cart1} grows trees by splitting each node, indexed by  the depth level $d$, with a probability $g(d)=\gamma/(1+d)^\alpha$, where $\alpha>1,\gamma\in (0,1)$ are tuning parameters. The splitting variable  is picked from  $\{x_1,\dots,x_p\}$ {\sl uniformly} (i.e. with a fixed probability $\theta_i=1/p$). Recently, \cite{linero2016} proposed an adaptive variant of this prior by placing a sparsity-inducing Dirichlet prior on the splitting proportions $\theta_i$.
This prior uses fewer variables in the tree construction and thereby is more reluctant to overfit. In another popular Bayesian CART implementation, \cite{cart2} suggest directly placing a prior on $K$ and a conditionally uniform prior on  tree topologies with $K$ bottom leaves. Again, in its original form, this prior will likely suffer from the curse of dimensionality, failing to harvest the intrinsic lower-dimensional structure. Here, we propose a fix to this problem. To make the Bayesian CART prior of \cite{cart2} appropriate for high-dimensional setups, we  propose a {\sl spike-and-tree} variant by injecting one more layer: a complexity prior over the active set of predictors.



Bayesian models for feature selection have traditionally involved a hierarchy of priors over subset sizes $q=|\mS|$ and subsets $\mS\subseteq\{1,\dots,p\}$ \cite{castillo_vdv, castillo_vdv2}.
{Instead of} modeling the mean outcome as a linear functional of active predictors $\x_{\mS}$, here we grow a tree from $\x_{\mS}$.
We begin by treating $q_{0}$ as unknown with an exponentially decaying prior  \cite{castillo_vdv}
\begin{equation}\label{prior:dim}
\pi(q)\,\propto\, c^{-q}\,p^{-a\, q},\, q=0,1,\dots, p, \quad\text{ for some} \quad a,c>0. \tag{T1}
\end{equation}
Next, given the dimensionality $q$, we assume that all  $p\choose q$ subsets $\mS$ of  $q=|\mS|$  covariates  are a-priori equally likely, i.e.
\begin{equation}\label{prior:subset}
\quad\quad\pi(\mS\C q)=1/{p\choose q}.\tag{T2}
\end{equation}
Given $q$ and the feature set $\mathcal{S}$, a tree is grown by splitting  at least once on every coordinate inside $\mS$, all the way down to  $K$ terminal nodes. If we knew $\alpha$ and $q_0$, the optimal choice of $K$ would be $K\asymp n^{q_0/(2\alpha+q_0)}$, for which the actual minimax rate could be achieved. For the more practical case when $\alpha$ and $q_0$  are both unknown,  we shall assume that $K$ arrived from a prior 
$\pi(k\C q)$.  As noted by \cite{coram}, the tail behavior of $\pi(k\C q)$ is critical for controlling  the vulnerability/resilience to overfitting. 
We consider the Poisson distribution (constrained to $\N\backslash\{0\}$), as suggested by \cite{cart2} in their Bayesian CART implementation. Namely, for $k\in \N\backslash\{0\}$ we have
\begin{equation}\label{prior:K}
\pi(K)\,=\, \frac{\lambda^{K}}{(\e^\lambda-1)K!},\,K=1,2,\dots,\quad\text{for some}\quad \lambda\in\R.\tag{T3}
\end{equation}
 For its practical implementation, one would truncate its support to 
 the maximum number of splits that can be made with $n$ observations. 
  When $\lambda$ is small,  \eqref{prior:K} is concentrated on models with smaller complexity where overfitting does not occur. Increasing $\lambda$ leads to smearing the prior mass over partitions with more jumps.
Similar complexity priors with an exponential decay $\e^{-C\, k\log k}$ have been deployed previously in nonparametric problems \cite{lian2007,coram,liu2015,cart2,vdp_rockova}. 

Given $q,\mathcal{S}$ and $K$, we assign a uniform prior over {\sl valid} tree topologies $\mT=\{\Omega_k\}_{k=1}^K\in\mV^K_\mS$, i.e.
\begin{equation}\label{prior:partition}
\pi(\mT\C \mathcal{S}, K)=\frac{1}{\Delta(\mV_\mS^K)}\1\left(\mT\in\mV^K_\mS\right)\tag{T4}.
\end{equation}
Similar constraints on trees where each terminal node is assigned a minimal number of data points have been implemented in stochastic search algorithms \cite{cart2}.
{At the very least, we can choose $\bar{C}=1$ in  \eqref{eq:balance} }, merely requiring that the cells be non-empty.
Finally, given the partition of size $K$, we assign an iid Gaussian prior on the step heights (similarly as in \cite{cart1}) 
\begin{equation}\label{prior:beta2}
\pi(\b\C K)=\prod_{k=1}^K\phi(\beta_k;1)\tag{T5}.
\end{equation}

\begin{remark}
The name {\sl spike-and-tree} prior deserves a bit of explanation. It follows from the fact that \eqref{prior:dim} and \eqref{prior:subset} will be satisfied if each covariate has a prior probability $\theta\sim\mathcal{B}(1,p^u)$  of contributing to the mean regression surface  for some $u>1$ \cite{castillo_vdv}.  Endowing each covariate $x_j$ with a Bernoulli indicator $\gamma_j$, where   $\Pi(\gamma_j=1\C\theta)=\theta$, and building a tree on $\mS=\{j:\gamma_j=1\}$, one obtains a mixture prior on $f(\x)$ that pertains  to spike-and-slab variable selection. Here, the slab is a {\sl tree} prior built on active covariates rather than an independent product prior on active regression coefficients. This hierarchical construction has distinct advantages for variable selection. In linear regression, it is  customary  to select variables by thresholding marginal posterior inclusion probabilities $\Pi(\gamma_j=1\C\Y^{(n)})$. These will be available also under our {\sl spike-and-tree} construction.
Inspecting the posterior probabilities $\Pi(\gamma_j=1\C\Y^{(n)})$ obtained with our hierarchical tree prior will be a new avenue for conducting variable selection in Bayesian CART and BART, an alternative to \cite{kapelner2015}. Thus, our prior  has  important practical implications for performing principled model-free variable selection.
\end{remark}

\subsection{Posterior Concentration for Bayesian CART}\label{sec:post_conc_bcart}
 The difficulty in properly analyzing Bayesian CART stems from the {combinatorial richness} of the prior  that makes it  less tractable analytically.  By building on our developments from previous sections, we are now fully equipped to  present the first theoretical result concerning  this method.


The following theorem solidifies the optimality properties of Bayesian  CART  by showing that under the hierarchical prior \eqref{prior:dim}-\eqref{prior:beta2},  the posterior adapts to both smoothness and sparsity, concentrating at the (near) minimax rate that depends only on the number of strong covariates regardless of how many noise variables are present.  The near-minimaxity refers to an additional log factor. {The result holds for  sparse (high-dimensional) regimes,
where $p$ can be potentially much larger than $n$ and where $q_0\lesssim \log^{1/2} n$.}
We will denote by
$$
\mF_\mT= \bigcup_{q = 0}^{\infty} \bigcup_{K=1}^{\infty}  \bigcup_{\mS:|\mS|=q}\mathcal{F}(\mV^K_{\mS})
$$
the collection of all tree-structured step functions (with various tree sizes and split subsets) that can be obtained by partitioning $\mX$.

\begin{theorem}\label{thm:trees}
Assume $f_0\in \mathcal{H}^{\alpha}_{p}\cap \mathcal{C}(\mS_0)$  with { $0<\alpha\leq 1$} and $0<q_0=|\mS_0|$ such that
{$q_{0}\|f_0\|_{\Ha}\lesssim \log^{1/2} n$ and $\|f_0\|_\infty\lesssim\log^{1/2} n$}. Moreover, we assume that $\log p\lesssim n^{q_{0}/(2\alpha+q_{0})}$ and that $\mX$ is $(M,\mS_0)$-regular. We endow $\mF_\mT$ with priors \eqref{prior:dim}-\eqref{prior:beta2}.  
Then with {$\varepsilon_n=n^{-\alpha/(2\alpha+q_{0})}\log^{1/2} n$} we have 
\begin{equation*}
\Pi\left( f \in \mathcal{F}_\mT : \|f_0 - f\|_n > M_n\,\varepsilon_n \mid \Y^{(n)}\right) \to 0,
\end{equation*}
 for any $M_n \to \infty$  in $P_0^n$-probability, as $n,p \to \infty$. 
\end{theorem}
\begin{proof}
Section \ref{proof:thm:trees}.
\end{proof}
\vspace{-0.5cm}
{
\begin{remark}\label{remark_beta}
It is useful to note that Theorem \ref{thm:trees} holds also when $p<n$ and when $q_0$ is fixed as $n\rightarrow\infty$. {When $q_0$ {is} fixed, however,  the assumptions $\|f_0\|_\infty\lesssim \log^{1/2}n$ and $q_0\|f_0\|_{\Ha}\lesssim \log^{1/2}n$ can be omitted. The first assumption is needed here to make sure that the step sizes of an approximating $k$-$d$ tree are well behaved when $q_0\rightarrow\infty$.} The result holds for a bit slower rate $\varepsilon_n= n^{-\alpha/(2\alpha+q_{0})}\log^{\beta} n$ with $\beta> 1/2$ under slightly relaxed assumptions
$q_{0}\|f_0\|_{\Ha} \lesssim \log^{\beta} n$ and $\|f_0\|_\infty\lesssim\log^{\beta} n$.
\end{remark}
}
The  assumption of a regular design is an inevitable consequence of treating $\x$'s as fixed. As noted by \cite{biau2016} in their study of random forests,  point-wise consistency results have been complicated by the difficulty in controlling  local (cell) diameters. The regularity assumption guarantees this control and is apt to be satisfied for most realizations of $\x$ from reasonable distributions on $[0,1]^p$. 
{The following Corollary certifies that Bayesian CART, under a suitable complexity prior on the number of terminal nodes, is reluctant to overfit. {This is seen} from the behavior of the posterior, which concentrates on values $K$ that are only a constant multiple  larger than the optimal oracle value $n^{q_0/(2\alpha+q_0)}$.}



\begin{corollary}\label{corollary1}(Bayesian regression trees do not overfit.)
Under the assumptions of Theorem \ref{thm:trees} {with $0<\alpha\leq 1$} we have
$$
\Pi\left(K> C_k n^{q_0/(2\alpha+q_0)} \mid \Y^{(n)}\right) \to 0
$$ 
in $P^n_0$-probability for a suitable constant $C_k>0$.
\end{corollary}
\begin{proof}
This statement follows from Lemma 1 of \cite{ghosal_vdv} and holds upon the satisfaction of the condition $\Pi(K>C_k n^{q_0/(2\alpha+q_0)})=o(\e^{-(d+2)\,n\varepsilon_n^2})$ for some suitably large $d>2$. This is shown in Section \ref{sec:proof:tree:eq:remain}.
\end{proof}

{Corollary 4.1 also reveals a fundamental  limitation of  trees (step functions) in recovering smoother functions.  To see this, consider $f_0(x) = x$ which possesses H\"{o}lder smoothness $1$ and, by Corollary 4.1,  will thus be approximated by trees with at most $n^{1/3}$ leaves (up to multiplicative constants) with high posterior probability. However, $f_0$ is also in $\mathcal{C}^2$ and the approximation error by a regular histogram with $n^{1/3}$ leaves will be at least of the order $n^{-1/3}$ which is too large to achieve the minimax rate of $n^{-5/2}$ over $\mathcal{C}^2$. }

\begin{remark}\label{remark:broad_class}(Bayesian CART \`{a} la Chipman, George and McCulloch \cite{cart1})
A closer look at the proof reveals that  Theorem \ref{thm:trees} also holds for priors $\pi(k\C q)$ such that $\e^{-c\, k\log{k}} \lesssim \pi(k\C q) \lesssim e^{-c\, k\log(n)^\eta}$ for some $0<\eta<1$. At the lower end are the complexity priors deployed in similar contexts \cite{scricciolo,vdp_rockova}. The Bayesian CART prior of \cite{cart1}, which is deployed in BART (and discussed in Section \ref{sec:spike_tree})  with $\alpha=0$ is at the upper end.
To see this, note that $\alpha=0$ corresponds to a homogeneous Galton-Watson (GW) process, where the number of terminal nodes $K$ satisfies
$\Pi(K>k)\leq \e^{-(k-0.5)\log[(1-\gamma)/\gamma]}$.
{With $\gamma=c/n$ for some $c<n$ we have $\Pi(K>k)\propto \e^{-c \, k\log n}.$ Going further,  the tail bound for the {\sl heterogeneous} GW process (with $\alpha\neq0)$ satisfies $\Pi(K>k)\leq\e^{-C_K k\log k}$ for some $C_K>0$ under a suitable modification of the split probability (i.e. $\pi(d)\propto a^{d}$ for some $1/n<a\leq 1/2$), as shown formally in Ro\v{c}kov\'{a} and Saha \cite{rockova_saha}.
They also show that the Bayesian CART posterior \`{a} la Chipman et al. \cite{cart1} concentrates at the optimal rate when $S_0$ is known. Endowed with the spike-and-slab wrapper, Theorem \ref{thm:trees} can be thus extended to the actual Bayesian CART prior deployed in BART.}


\end{remark}

\section{Tree Ensembles}\label{sec:treeensembles}
Combining multiple trees through additive aggregation has proved to be remarkably effective for enhancing prediction \cite{breiman2,bart}. This section offers new theoretical insights into the mechanics behind the Bayesian variants of such tree ensemble methods. Our approach rests on a detailed analysis of the {\sl collective behavior} of partitioning cells generated by individual trees. 
We will see that the overall performance is affected not only by the quality of single trees but also how well they can collaborate \cite{breiman_book}.


\subsection{Bayesian Additive Regression Trees}\label{sec:bart}
Additive regression trees  grow an ensemble predictor by binding together  $T$ tree-shaped regressors. For subsets $\bm{\mS}=\{\mS^1,\dots,\mS^T\}$ and tree sizes $\bm{K}=(K^1,\dots,K^T)'$, we define a sum-of-trees model (forest) as 
\begin{equation}\label{tree_ensemble}
f_{\mathcal{E},\mathcal{B}}(\x)=\sum_{t=1}^Tf_{\mT^t,\b^t}(\x)=\sum_{t=1}^T\sum_{k=1}^{K^t}\beta_k^{t}\,\mathbb{I}(\x\in{\Omega}_k^{t}),
\end{equation}
where $f_{\mT^t,\b^t}\in\mF(\mV_{\mS^t}^{K^t})$ is a tree learner associated with step sizes $\b^t\in\R^{K^t}$. Each learner is allowed to use different splitting variables $\mS^t$ and different number of bottom nodes $K^t$. With $\mE=\{\mT^1,\dots,\mT^T\}$ we denote an ensemble of tree  partitions 
and with $\mB=(\b^{1\prime},\dots,\b^{T\prime})'\in\R^{T\bar{K}}$ the terminal node parameters, where $\bar{K}=\frac{1}{T}\sum_{t=1}^TK^t$. 


Sum-of-trees models offer an improved representation flexibility by chopping up the predictor space into  more refined segmentations. These segmentations are obtained by superimposing multiple tree partitions, yielding what we define below as a {\sl global} partition.
\begin{definition}\label{def:global}
For a partition ensemble $\mE$, we define  a  {\sl global partition}
$$
\wt{\mT}(\mE)=\{\wt{\Omega}_k\}_{k=1}^{K(\mE)}
$$
as the partition obtained by merging  all cuts in $\mT^1,\dots,\mT^T$. We refer to $\wt{\Omega}_k$'s as {\sl global} cells in the ensemble.
\end{definition}
The concept of the global partition can be better understood from  Figure \ref{pic:tree_assembly}, where splits from $T=2$ trees (each having $3$ leaves) are merged  to obtain a global partition with ${K}(\mE)=7$ global cells. Generally, the global partition $\mathcal{E}$ itself is {\sl not} necessarily a tree and  can have as many as  $K(\mE)\leq \prod_{t=1}^TK^t$ cells. This upper bound  can be attained if each  trees splits $K^t-1$ times on a single variable, where each tree uses a different one. Without loss of generality,  the global partition will be assumed to have non-empty cells. This requirement can be met by merging vacuous cells with their nonempty neighbors.

\begin{figure}[!t]
\scalebox{0.35}{\includegraphics{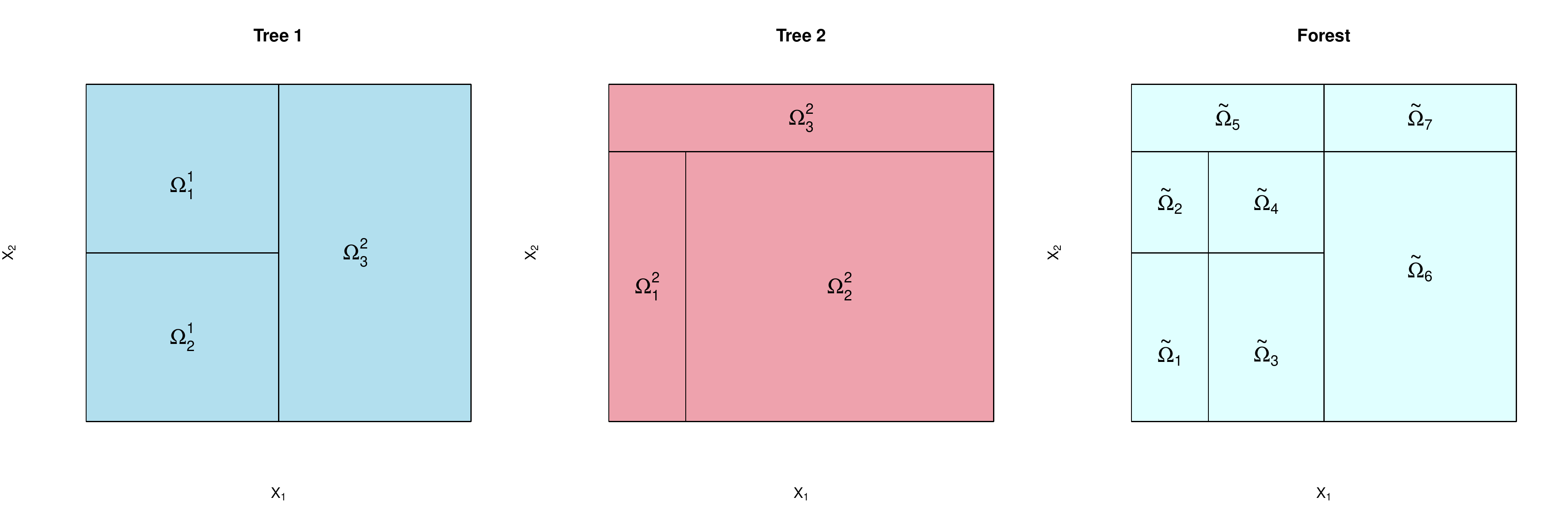}}
\caption{Ensemble of tree partitions and their global partition.}\label{pic:tree_assembly}
\end{figure}

Bayesian additive regression trees were conceived as  a collection of weak learners that capture {\sl different} aspects of the predictor space \cite{bart}. To characterize the amount of diversity/correlation between trees in the ensemble, we introduce the so-called stretching matrix. 

\begin{definition}\label{def:concat}
For a partition ensemble $\mE=\{\mT^1,\dots,\mT^T\}$, we define  the {\sl stretching matrix}  $\bm{A}(\mE)=(a_{ij})_{i,j}$ as follows:  for each $1\leq i\leq K(\mE)$ and  $1\leq j\leq  T\bar{K}$
we have
\begin{equation}\label{stretch}
a_{ij}=\1\{\widetilde{\Omega}_i\cap \Omega^t_m\neq \emptyset \},
\end{equation}
where  $1\leq t\leq T$ and $1\leq m\leq K^t$ are such that $j=\sum_{l=1}^{t-1}K^l+m$ and where $\Omega^t_m$ is the $m^{th}$ (local) cell in the $t^{th}$ tree $\mT^t$.
\end{definition}
Each row of the stretching matrix $\A(\mE)$ corresponds to one global cell and each column to one local cell.
The row entries sum to $T$, indicating which local cells overlap with that global cell 
(as shown in Figure \ref{pic:conc}  for partitions from  Figure \ref{pic:tree_assembly}).
To further characterize the pattern of overlap between trees, we introduce the $(T\bar{K})\times (T\bar{K})$ Gram matrix 
\begin{equation}\label{gram}
\wt{\bm{A}}(\mE)\equiv\bm{A}(\mE)'\bm{A}(\mE)=\left(\wt{a}_{ij}\right)_{i,j}.
\end{equation}
The off-diagonal elements  measure the ``similarity" between local cells, say $\Omega^{t}_j$ and $\Omega^{u}_k$, in terms of  the number of global cells that they share. More formally,  let
$r[\wt{\mT}(\mE),V]=|\{\wt{\Omega}_k: \wt{\Omega}_k\cap V\neq\emptyset\}|$ be the restricted cell count \cite{nobel}, measuring the number of global cells that intersect with a compact set $V$. For $i=\sum_{s=1}^{t-1}K^s+j$ and $l=\sum_{s=1}^{u-1}K^s+k$ we can write $\wt{a}_{il}=r[\wt{\mT}(\mE),\Omega^{t}_j\cap \Omega^{u}_k]$.  
 Small off-diagonal entries  indicate less overlap, where the individual trees capture more diverse aspects of the predictor space. The diagonal elements, on the other hand,  quantify the ``persistence" of each local cell,  say $\Omega_k^{u}$, counting the number of global cells it stretches over.
 More formally, for $i=\sum_{s=1}^{u-1}K^s+k$ we have $\wt{a}_{ii}=r[\wt{\mT}(\mE),\Omega_k^{u}]$.  The amount of diversity (or tree dis-similarity) in the  ensemble can be quantified with eigenvalues of $\wt{\bm{A}}(\mE)$.  We denote by $\lambda_{min}(\mE)$ (resp. $\lambda_{max}(\mE)$) the minimal (resp. maximal) singular values of ${\A}(\mE)$ (i.e. square roots of extremal nonzero eigenvalues of $\widetilde \A(\mE)$). If some trees in the ensemble are redundant,  the conditioning number $\kappa(\mE)\equiv\lambda_{max}(\mE)/\lambda_{min}(\mE)$ will be large. 
 The idea of diversifying trees was originally introduced  by Breiman \cite{breiman2} via subsampling.  
   {
   One could, in principle,  impose a restriction on $\lambda_{min}(\mE)$ in the prior to encourage the trees to collaborate and diversify. However, this is  {\sl not} required for our theoretical study. We will focus on the so-called valid ensembles which consist of valid trees.


\begin{figure}
{\scriptsize $$
 \bm{A}(\mathcal{E})=
 \begin{blockarray}{ c c c c c c c}
   &  \color{Gray}\Omega_1^1 &\color{Gray}\Omega_2^1& \color{Gray}\Omega_3^1 &\color{Gray}\Omega_1^2&\color{Gray} \Omega_2^2 &\color{Gray} \Omega_3^2 \\      
\begin{block}{c\Left{}{(\mkern1mu} c c c c c c<{\mkern1mu})}
 \color{Gray}  \widetilde{\Omega}_1 &0 &1 &0 &1 &0 &0 \\
\color{Gray}  \widetilde{\Omega}_2 &1  &0&0 &1 &0 &0 \\
\color{Gray}  \widetilde{\Omega}_3 &0 &1 &0 &0 &1 & 0\\
\color{Gray}  \widetilde{\Omega}_4 &1 &0 &0 &0 &1 & 0\\
\color{Gray}  \widetilde{\Omega}_5 &1 &0 &0 &0 &0 & 1\\
\color{Gray}  \widetilde{\Omega}_6 &0 &0 &1 &0 &1 & 0\\
\color{Gray}  \widetilde{\Omega}_7 &0 &0 &1 &0 &0 & 1\\
 \end{block}
  \end{blockarray},\quad
  \bm{A}(\mathcal{E})'\bm{A}(\mathcal{E})=
  \left(
  \begin{array}{c c c | c c c}
3&0&0&1&1&1\\
0&2&0&1&1&0\\
0&0&2&0&1&1\\
\hline
1&1&0&2&0&0\\
1&1&1&0&3&0\\
1&0&1&0&0&2\\
  \end{array}
  \right)
$$}
\vspace{-0.5cm}
\caption{Stretching matrix $\A(\mE)$ and the ``Gram matrix" $\A(\mE)'\A(\mE)$ from Figure \ref{pic:tree_assembly}.}\label{pic:conc}
\end{figure}

\begin{definition}\label{def:validensemble}
An ensemble $\mE=\{\mT^1,\dots,\mT^T\}$
is {\sl valid} if 
 each $\mT^t$ is {\sl valid} according to Definition \ref{ass:balance}.
For tree sizes ${\bm K}=(K^1,\dots, K^T)'$ and subsets $\bm{\mS}=\{\mS^1,\dots,\mS^T\}$, we denote the set of all valid ensembles by $\mV\mE_{\bmS}^{\bm K}$. 
\end{definition}}

The representation flexibility of additive trees also pertains to jump sizes. The global step size coefficients under additive trees are intertwined due to the tree overlap. This can be seen from the following, more compact, representation of \eqref{tree_ensemble}:
$$
f_{\mathcal{E},\mathcal{B}}(\x)=\sum_{k=1}^{K(\mE)} \bar{\beta}_k\,\mathbb{I}(\x\in{\wt{\Omega}}_k)\quad\text{with}\quad \bar{\beta}_k\equiv\sum_{t=1}^T\sum_{l=1}^{K^t}\beta_l^{t}\mathbb{I}(\Omega_l^{t}\cap{\wt{\Omega}}_k\neq \emptyset),
$$
where $\bar{\b}=(\bar{\beta}_1,\dots,\bar{\beta}_{K(\mE)})'\in\R^{K(\mE)}$ are aggregated step sizes. A closer look reveals the following key connection  between  $\bar{\b}$  and $\mB$ 
\begin{equation}\label{bar_mapping}
\bar{\b}=\bm{A}(\mathcal{E})\mB,
\end{equation}
where  $\bm{A}(\mathcal{E})$ is the stretching matrix defined in \eqref{stretch}.
This link  unfolds the theoretical analysis of tree ensembles using  tools that we have already developed for single trees. Note that the condition number $\kappa(\mE)$ determines how much the relative change in $\bar\b$ influences the relative change in $\mB$. 

The mapping \eqref{bar_mapping} can be in principle many-to-one in the sense that many tree-structured step functions can sum towards the same target \eqref{tree_ensemble}. 
Such over-parametrization occurs, for instance, when ${K}(\mE)<T\bar{K}$ or, more generally, when ${\A}(\mE)$ has zero eigenvalues. This redundancy is not entirely unwanted  and, in fact, it endows sum-of-trees models with a lot of modeling freedom. 


We now formally define the space of approximating additive trees. For  variable  sets $\bmS=\{\mS^1,\dots,\mS^T\}$ and a vector of tree sizes  $\bm{K}=(K^1,\dots, K^T)'$, we denote by
\begin{equation}\label{eq:mfv}
\mF(\mV\mE_{\bmS}^{\bm{K}})=\left\{f_{\mE,\mB}:[0,1]^p\rightarrow\R:f_{\mE,\mB}(\x)=\sum_{t=1}^Tf_{\mT^t,\b^t}(\x);\mE\in\mV\mE_{\bmS}^{\bm K},\b^t\in\R^{K^t} \right\}
\end{equation}
the set of all additive tree step functions supported on {\sl valid} ensembles  $\mV\mE_{\bmS}^{\bm K}$.
The union of these over the number of trees $T$, all possible sets $\bmS$ of sizes $\bm{q}=(q^1,\dots,q^T)'$ and tree sizes $\bm{K}$ gives rise to
\begin{equation}\label{mFE}
\mF_\mE=\bigcup_{T=1}^\infty\bigcup\limits_{\substack{\bm{q}}} \bigcup\limits_{\substack{\bmS:|\mS^t|=q^t}}
\bigcup\limits_{\substack{\bm{K}}} \mF(\mV\mE_{\bmS}^{\bm{K}}),
\end{equation}
our approximating space of additive regression tree functions.

\subsection{Additive Regression Trees are Adaptive}\label{sec:ensembles}
This section provides an interesting  initial perspective on the behavior of Bayesian additive regression trees in Regime 1. We will continue with the more general Regime 2 in the next section.
We will focus on a  variant of  the popular Bayesian Additive Regression Trees (BART) model \cite{bart}, modified in three ways. First, the tree prior will be according to   \cite{cart2} rather than \cite{cart1}. The second modification is that the trees are built on the {\sl same set of variables}, endowed with a subset selection prior construction. In the next section, we allow for the fully  general case where each tree builds on a potentially different set of variables. Third, rather than fixing the number of trees, we endow $T$ with a prior distribution.

We will see that having a good control of the regression function variation inside each {\sl global} cell together with a good choice of the prior on the total number of leaves $\sum K^t$ will be sufficient to ensure optimal behavior. The approximation ability of tree ensembles hinges on the diameter of the {\sl global} partition. Each tree partition does not  need to have a small diameter (i.e. can be a weak learner), as long as the global one does. 
An important building block in our proof will be the construction of a single tree ensemble that can approximate well. As will be shown in Lemma \ref{lemma:approx:ensemble}, we can construct such ensemble by first finding a single $k$-$d$ tree from Lemma \ref{lemma:approx} (a strong learner) and then redistributing the cuts among small trees (weak learners) in a way that the global partition is exactly equal to  the $k$-$d$ tree.
An example of this deconstruction is depicted in Figure \ref{bayestrees},   where a full symmetric tree from Figure \ref{bayestrees2}, say $\widehat{\mT}$, has been trimmed into many smaller imbalanced trees which add up towards $\widehat{\mT}$. More details on this decomposition are in the proof of Lemma \ref{lemma:approx:ensemble}. 

\begin{figure}
    \subfigure[$\mT^1$]{
    \begin{minipage}[b]{0.2\textwidth}
   
\scalebox{0.7}{
\xymatrix@-1pc{
 & & &   *++[o][F-]\txt{}  \ar@{-}[dl] \ar@{-}[dr] & \\
 & &     *++[o][F-]{\txt{}}  \ar@{-}[dl]  \ar@{-}[dr] &  &  *++[o][F]\txt{}\\
&       *++[o][F-]\txt{{}}  \ar@{-}[dl] \ar@{-}[dr]   &    & *++[o][F]\txt{}&\\
*++[o][F]\txt{} &    &    *++[o][F]\txt{}& &\\
}}
      \end{minipage}}
 \subfigure[$\mT^2$]{
    \begin{minipage}[b]{0.2\textwidth}

\scalebox{0.7}{
\xymatrix@-1pc{
 & & &   *++[o][F-]\txt{}  \ar@{-}[dl] \ar@{-}[dr] & \\
 & &     *++[o][F-]{\txt{}}  \ar@{-}[dl]  \ar@{-}[dr] &  &  *++[o][F]\txt{}\\
&       *++[o][F-]\txt{{}}   &    & *++[o][F]\txt{} \ar@{-}[dl] \ar@{-}[dr]  &\\
&&*++[o][F]\txt{} &    &    *++[o][F]\txt{} \\
}}
      \end{minipage}}
 \subfigure[$\mT^3$]{
    \begin{minipage}[b]{0.2\textwidth}
\scalebox{0.7}{
\xymatrix@-1pc{
 & & &   *++[o][F-]\txt{}  \ar@{-}[dl] \ar@{-}[dr] && \\
 & &     *++[o][F-]{\txt{}}&  &  *++[o][F]\txt{}  \ar@{-}[dl]\ar@{-}[dr]& \\
&         &    & *++[o][F]\txt{} \ar@{-}[dl] \ar@{-}[dr]  & & *++[o][F-]\txt{{}} \\
&&*++[o][F]\txt{} &    &    *++[o][F]\txt{} &\\
}}
\end{minipage}}
       \subfigure[$\mT^4$]{
    \begin{minipage}[b]{0.2\textwidth}
\scalebox{0.7}{
\xymatrix@-1pc{
 & & &   *++[o][F-]\txt{}  \ar@{-}[dl] \ar@{-}[dr] && \\
 & &     *++[o][F-]{\txt{}}&  &  *++[o][F]\txt{}  \ar@{-}[dl]\ar@{-}[dr]& \\
&         &    & *++[o][F]\txt{}   & & *++[o][F-]\txt{{}}\ar@{-}[dl] \ar@{-}[dr] \\
&& &&*++[o][F]\txt{} &    &    *++[o][F]\txt{} \\
}}
\end{minipage}}
    \caption{\scriptsize Trees in the approximating ensemble $\wh{\mE}$.}
    \label{bayestrees}
\end{figure}

\begin{figure}

\scalebox{0.7}{
\xymatrix@-1pc{
& & & &  & &  *++[o][F-]\txt{}  \ar@{-}[dlll] \ar@{-}[drrr] &&& & & &\\
 & & &     *++[o][F-]{\txt{}}\ar@{-}[dll] \ar@{-}[dr]& & && &&  *++[o][F]\txt{}  \ar@{-}[dl]\ar@{-}[drr]& &&\\
&      *++[o][F]\txt{} \ar@{-}[dl] \ar@{-}[dr]   && & *++[o][F-]{\txt{}}\ar@{-}[dl] \ar@{-}[dr] &   & &  &  *++[o][F-]{\txt{}}\ar@{-}[dl] \ar@{-}[dr] & & & *++[o][F-]\txt{{}}\ar@{-}[dl] \ar@{-}[dr]& \\
*++[o][F]\txt{} & & *++[o][F]\txt{} & *++[o][F]\txt{}& &*++[o][F]\txt{}  &&*++[o][F]\txt{}  &&  *++[o][F]\txt{}  &   *++[o][F]\txt{}  & & *++[o][F]\txt{}    \\
}}
   \caption{\scriptsize  Approximating $k$-$d$ tree partition $\wh{\mT}=\widetilde{\mT}(\wh{\mE})$}
    \label{bayestrees2}
\end{figure}
\CompileMatrices

The following theorem is an ensemble variant of Theorem \ref{thm:trees} which will serve as a useful stepping stone towards the full-fledged  result presented in the next section. 
Instead of approximating $f_0\in\Ha_p\cap\mathcal{C}(\mS_0)$  with one large tree in Regime 1, we build a forest made up of many smaller trees (weak learners). The first two layers of the prior \eqref{prior:dim} and \eqref{prior:subset} are the same. Next, we assign a prior on the number of trees
 \begin{equation}\label{prior:T}
\pi(T)\propto \e^{-C_T\, T}\quad\text{for}\quad T\in\N\backslash\{0\}\quad\text{with}\quad C_T>\log 2,\tag{T3*}
\end{equation}
which is sufficiently diffuse so as to promote ensembles with many trees. Next, conditionally on $T$, we assign a prior on $\bm{K}=(K^1,\dots, K^T)'$ as an independent product of Poisson distributions \eqref{prior:K}. An important distinction between the Poisson prior deployed for single trees in Section \ref{sec:singletree} and the one deployed here for ensembles  is that the hyper-parameter  depends on $T$. Namely, our prior on leaves satisfies
\begin{equation}\label{prior:KK}
\pi(\bm{K}\C T)\,=\, \prod_{t=1}^T\frac{(\lambda/T)^{K^t}}{(\e^{\lambda/T}-1)K^t!}\quad\text{for}\quad K^t\in\N\backslash\{0\} \tag{T4*}.
\end{equation}
As a prior on $\mE$, given $(\mS,\bm{K})$, we use the uniform prior over valid ensembles
\begin{equation}\label{prior:partition}
\pi(\mE\C \mathcal{S}, \bm{K})=\frac{1}{\Delta(\mV\mE_\mS^{\bm{K}})}\1\left( \mE\in \mV\mE_\mS^{\bm{K}} \right)\tag{T5*},
\end{equation}
where  $\Delta(\mV\mE_\mS^{\bm{K}})$ is the overall number of ensembles consisting of $T$ valid trees that can be obtained by splitting on data points $\mX$.
Recall that throughout this section, all trees are constrained to use split variables in the same set $\mS$. To mark this difference, we have denoted the partition ensembles with $\mV\mE_{\mS}^{\bm K}$ instead of $\mV\mE_{\bmS}^{\bm K}$. 

Finally,  given $(T,\bm{K})$, we assign an iid Gaussian prior on the step heights $\mB\in\R^{\sum_{t=1}^TK^t}$ with a variance $1/T$ (as suggested by \cite{bart})
\begin{equation}\label{prior:beta}
\pi(\mB\C T,\bm{K})=\prod_{t=1}^T\prod_{k=1}^{K^t}\phi(\beta_k^t;1/T)\tag{T6*}.
\end{equation}
The following theorem shows that additive regression trees, in combination with a subset selection prior,  can nicely adapt to the ambient dimensionality and smoothness, 
 {\sl also} achieving the optimal  concentration rate in Regime $1$.

\begin{theorem}\label{thm:weak_learner}
Assume $f_0\in \mathcal{H}^{\alpha}_{p}\cap \mathcal{C}(\mS_0)$  with {$0<\alpha\leq 1$} and $0<q_0=|\mS_0|$ such that
{$q_{0}\|f_0\|_{\Ha} \lesssim \log  n$   and $\|f_0\|_\infty\lesssim \log n$}. Moreover, we assume $\log p\lesssim n^{q_{0}/(2\alpha+q_{0})}$ and that $\mX$ is $(M,\mS_0)$-regular. We endow $\mF_\mE$ with priors \eqref{prior:dim},\eqref{prior:subset},\eqref{prior:T}-\eqref{prior:beta}.   With {$\varepsilon_n=n^{-\alpha/(2\alpha+q_{0})}\log n$} we have 
 \begin{equation*}
\Pi\left( f \in \mathcal{F}_\mE : \|f_0 - f\|_n > M_n\,\varepsilon_n \mid \Y^{(n)}\right) \to 0,
\end{equation*}
 for any $M_n \to \infty$  in $P_0^n$-probability, as $n,p \to \infty$. 
\end{theorem}
\begin{proof}
Supplemental  Material (Section 2).
\end{proof}
Similarly as for single trees, we obtain the following corollary which states that  the posterior concentrates on ensembles whose {\sl overall number of leaves} is not much larger than the optimal value $n^{q_0/(2\alpha+q_0)}\log n$.
\begin{corollary}\label{corollary2}
Under the assumptions of Theorem \ref{thm:weak_learner}  {with $0<\alpha\leq 1$} we have 
$$
\Pi\left((T,\bm{K}): \sum_{t=1}^T K^t> C_k n^{q_0/(2\alpha+q_0)}\log n \mid \Y^{(n)}\right) \to 0
$$ 
in $P_0^n$-probability, as $n,p \to \infty$, for a suitable constant $C_k>0$.
\end{corollary}
\begin{proof}
This follows from Lemma 1 of \cite{ghosal_vdv}  and Section \ref{sec:last_one}.
\end{proof}
{Corollary \ref{corollary2} shows that the posterior distribution rewards either many weak learners or a few strong ones. }
The compromise between the two is regulated by the prior \eqref{prior:T}, where stronger shrinkage (i.e. larger $C_T$) will result in fewer trees. This corollary provides an important theoretical justification for why Bayesian additive regression trees have been so resilient to overfitting in practice. 
{\begin{remark}
The dependence on $T$ in the Poisson prior (T4$^\star$)  works nicely in tandem with the exponential prior \eqref{prior:T}. Removing  $T$ from (T4$^\star$) would have to be balanced with a bit stronger prior $\pi(T)$.  Theorem \ref{thm:weak_learner} also holds with $T$ fixed (with slight modifications of the proof). The prior $\pi(T)$ will be instrumental in the additive case (next section). The dependence on $T$ in \eqref{prior:beta}, though recommended in practice \cite{bart}, is not needed in Theorem \ref{thm:weak_learner}.
\end{remark}}
\vspace{-0.5cm}

\section{Tree Ensembles in Additive Regression}\label{sec:additiveensembles}
In Section \ref{sec:post_conc_bcart} we have shown that the posterior distribution  under the Bayesian CART prior has optimal properties. However, it is now well known that the practical deployments of Bayesian CART suffer from poor MCMC mixing. Additive aggregations of single small trees \cite{bart} have proven to have far more superior mixing properties. One may wonder whether the benefits of additive trees are purely computational or whether there are some aspects that make them more attractive also theoretically. We  will address this fundamental question.

For estimating a single smooth function, we were not able to tell apart single trees from tree ensemble in terms of their convergence rate (besides perhaps a small difference in the log factor). They are both optimal in Regime 1.  Tree ensembles  are inherently additive and, as such, are well-equipped for approximating additive $f_0$ (Regime 2). Throughout this section, we assume 
\begin{equation}\label{f0:additive}
f_0(\x)=\sum_{t=1}^{T_0}f_0^t(\x),
\end{equation}
where  $f_0^t\in\mathcal{H}^{\alpha^t}_p\cap\mC(\mS_0^t)$. Note that each component $f_0^t$ depends only on a potentially very small subset $\mS_0^t$ of covariates, where $|\mS_0^t|=q_0^t$.
However, the additive structure allows $f_0$ to depend on a larger number of variables, say $q_0$, where $\max_{1\leq t\leq {T_0}}q_0^t\leq q_0 \leq \sum_{t=1}^{T_0}q_0^t$. The minimax rate $r_n^2$ for estimating $f_0$ in Regime 2 satisfies 
$C_1 \underline\varepsilon_n^2\leq r_n^2\leq C_2\bar\varepsilon_n^2$  \cite{yang},
where {
\begin{align*}
\underline\varepsilon_n^2&=\sum_{t=1}^{T_0}\left[\lambda^{t2}\left(\sqrt{n}\lambda^t\right)^{-4\alpha^t/(2\alpha^t+q_0^t)}+\frac{q_0^t}{n}\log\left(\frac{p}{\sum_t q_0^t}\right)\right],\\
\bar\varepsilon_n^2&= \sum_{t=1}^{T_0}\left[\lambda^{t2}\left(\sqrt{n}\lambda^t\right)^{-4\alpha^t/(2\alpha^t+q_0^t)}+\frac{q_0^t}{n}\log\left(\frac{p}{\min_t q_0^t}\right)\right]
\end{align*}
and  
where $\lambda^t$ is the H\"{o}lder norm of $f_0^t$}. The sparsity constraint  in Regime 2 is less strict than in Regime 1, where $q_0$ can be potentially larger than $\log n$ while still allowing for consistent estimation. For the isotropic case ($\alpha^t=\alpha$ and $q_0^t=\wt{q}_0$), single trees can achieve the slower rate $n^{-2\alpha/(2\alpha+q_0)}$ where $\widetilde{q}_0\leq q_0\leq T_0\widetilde{q}_0$.   As will be shown below, tree ensembles can achieve a faster rate {$\varepsilon_n^2=\sum_{t=1}^{T_0}(\varepsilon_n^{t})^2$ where  $\underline\varepsilon_n^2\lesssim \varepsilon_n^2\lesssim \bar\varepsilon_n^2$ and where
$(\varepsilon_n^{t})^2=\lambda^{t2}\left(\sqrt{n}\lambda^t\right)^{-4\alpha^t/(2\alpha^t+q_0^t)}+\frac{q_0^t}{n}\log\left(\frac{p}{q_0^t}\right)$.}

We will approximate $f_0$ with tree ensembles $f_{\mE,\mB}\in\mF_\mE$.  The ensembles here differ from the ones considered in Section \ref{sec:ensembles}.
The crucial difference is that now we allow each of the trees $\mT^t$ to depend on a {\sl different set of variables $\mS^t$}. 
Now we have a vector of subset sizes $\bm{q}=(q^1,\dots, q^T)'$ and a set of subsets $\bmS=\{\mS^1,\dots, \mS^T\}$, one for each tree.
We consider the following  independent product  variant of the complexity prior \eqref{prior:dim} 
\begin{equation}\label{prior:dim2}
\pi(\bm{q}\C T)\propto \prod_{t=1}^Tc^{-q^t}p^{-aq^t},\,q^t=0,1,\dots,p,\,\quad\text{for}\quad a>2 \tag{T1*}
\end{equation}
and a product prior variant of \eqref{prior:subset}, given $(T,\bm{q})$, 
\begin{equation}\label{prior:subset2}
\pi(\bmS\C T,\bm{q})\propto \prod_{t=1}^T1/{p\choose q^t} \tag{T2*}.
\end{equation}

The prior on the number of trees, the number of leaves, ensembles and step sizes is the same as in  \eqref{prior:T}, \eqref{prior:KK}, \eqref{prior:partition}.

We are now ready to present our final result showing that the posterior concentration for Bayesian additive regression trees is near-minimax rate optimal  when $f_0$ has an additive structure.

\begin{theorem}\label{thm:additivetrees}
Assume that $f_0$ is as in \eqref{f0:additive} where $f_0^t\in \mathcal{H}^{\alpha^t}_p\,\cap\, \mC(\mS_0^t)$  with {$0<\alpha^t\leq 1$} and $q_0^t=|\mS_0^t|$ such that
{$0<q_{0}^t \|f_0^t\|_{\mathcal{H}_{\alpha^t}}\lesssim\log^{1/2} n$} and  $\|f_0^t\|_\infty\lesssim\log^{1/2} n$. Moreover, assume  $\log p\lesssim \min\limits_{1 \leq t \leq T_0} n^{q_{0}^t/(2\alpha^t+q_{0}^t)}$ and that $\mX$ is $(M,\mS_0^t)$-regular for $1\leq t\leq T_0$. We endow $\mF_\mE$ with priors \eqref{prior:dim2}-\eqref{prior:beta}.  With {$\varepsilon_n^2= \sum_{t=1}^{T_0}  n^{-2\alpha^t/(2\alpha^t+q_{0}^t)}\log n$} and {$T_0\lesssim n$}, we have 
\begin{equation*}
\Pi\left( f \in \mathcal{F}_\mE : \|f_0 - f\|_n > M_n\, \varepsilon_n \mid \Y^{(n)}\right) \to 0,
\end{equation*}
 for any $M_n \to \infty$  in $P_0^n$-probability, as $n,p \to \infty$ .
\end{theorem}
\begin{proof}
Supplemental Material (Section 1).
\end{proof}
\vspace{-0.5cm}
{
\begin{remark}
Under the assumptions $\log p\lesssim n^{q_{0}^t/(2\alpha^t+q_{0}^t)}$ and $q_{0}^t \|f_0^t\|_{\mathcal{H}_{\alpha^t}}\lesssim\log^{1/2} n$, the second term $\frac{q_0^t}{n}\log\left(\frac{p}{q_0^t}\right)$ in $(\varepsilon_n^t)^2$  (defined earlier) is dominated by the first term $n^{-2\alpha^t/(2\alpha^t+q_0^t)}\log n$.  This is why the second term does not appear in  the rate $\varepsilon_n^2$ in Theorem \ref{thm:additivetrees}.
\end{remark}}

Failing to recognize the additive structure in $f_0$, single regression trees achieve the slower rate $n^{-\alpha/(2\alpha+q_0)}\log^{1/2} n$, according to Theorem \ref{thm:trees}. Theorem \ref{thm:additivetrees} thus provides an additional theoretical justification for Bayesian additive tree models suggesting their performance superiority over single trees when $f_0$ is additive. 


\subsection{Implementation Considerations}\label{sec:computation}
{Our priors differ from the widely used BART implementations in three ways: (1) we focus on the uniform prior of Denison et al. \cite{cart2}, (2) we assign a prior distribution on the number of trees and (c) we deploy the spike-and-slab wrapper. Implementations of our priors are feasible with some modifications of the existing software.
For the Bayesian CART prior that we analyze, Denison et al.  \cite{cart2} propose a reversible jump MCMC implementation. While this algorithm is different from BART,  the acceptance ratios in the Metropolis-Hastings step differ only very slightly.
Liu, Ro\v{c}kov\'{a} and Wang \cite{liu18} extended their sampler to the spike-and-tree (spike-and-forest) versions in two ways.  The first one is a Metropolis-Hasting strategy that  consists of joint sampling from variable subsets as well as trees (forests). As a faster alternative, they  proposed an approximate ABC sampling strategy based on data splitting (called {\sl ABC Bayesian Forests}).  

}

\section{Discussion}\label{sec:discussion}
In this work, we have laid down foundations for the theoretical study of Bayesian regression trees and their additive variants.
We have shown an optimal behavior of Bayesian CART, the first theoretical result  on this method.
We have developed several useful tools for analyzing additive regression trees (variants of the BART method), showing their optimal performance in both additive and non-additive regression. 
The smoothness order of  studied functions is restricted to values not exceeding one, a main limitation of our approach due to the fact that our approximations are piecewise constants \cite{ghosal_vdv, scricciolo,  Engel1994}. 
{ While in the one-dimensional case, step functions are not appealing estimators of a regression function that is thought to be smooth, methods like CART and BART are attractive and feasible solutions in complex high-dimensional data.} The limitation $\alpha\leq 1$ could be overcome by extending our approach to piecewise polynomials or kernels, an  elaboration that we leave for future investigation. {One such extension was recently proposed in a related paper by  Linero and Yang \cite{linero2018}. 
These authors obtained concentration results for a kernel method that can be regarded as a smooth variant of BART.   The results of \cite{linero2018} do not apply for single trees, only aggregates of kernels. In contrast, we study actual posteriors of single trees, as well as forests, and analyze sieves of step functions which are the essence of the actual BART method.

{While our priors do not exactly match the BART prior,  BART could be adapted to achieve the same optimality properties. The first modification is the splitting probability (as pointed out in Remark \ref{remark:broad_class}). The second modification is the spike-and-slab wrapper (as pointed out in Section \ref{sec:spike_tree}) or a modification of the prior on the split variables (as in \cite{linero2018}).
The prior distribution on the number of trees will only be beneficial in the additive model and is not needed when $f_0$ has one layer.

The assumption of a known $\sigma_0=\mathsf{Var}\, (\varepsilon_i)=1$ can be relaxed. It has been noted in the literature (e.g. \cite{vdv_zanten, dejonge2013}) that 
the general result of Ghosal and van der Vaart \cite{ghosal_vdv} (which we build upon) can be extended to the unknown $\sigma_0$ case. Such an extension was formally proved in Jonge and van Zanten (2013), who assume that $\sigma_0$ belongs to a compact interval $[a,b]$ and the prior $\pi(\sigma)$ concentrates on $[a,b]$. We could obtain our results under this restriction as well
by verifying suitably adapted conditions (2.1)-(2.3).
In related work, Yoo and Ghosal \cite{yoo2016} show optimal posterior concentration (in both $L_2$ and $L_\infty$ sense)  in non-parametric regression with unknown variance and B-spline tensor product priors. Next, they show that under an inverse-gamma prior, the posterior for $\sigma$ contracts at $\sigma_0$ at the same rate. Moreover, for any prior on $\sigma$ with positive and continuous density, the posterior of $\sigma$ is consistent. These results are obtained under the assumption that $f_0$ is uniformly bounded.
 We anticipate that  similar results will hold also for our priors when $f_0$ is uniformly bounded.\\
 }

{ \textbf{Acknowledgement}
We are grateful to the Associate Editor and two referees for helpful comments, and to Johannes Schmidt-Hieber for helpful discussions on the assumption $\alpha \leq 1$.
}

\section{Proof of Theorem \ref{thm:trees}}\label{proof:thm:trees}
Our approach consists of establishing conditions \eqref{eq:entropy}, \eqref{eq:prior} and \eqref{eq:remain} for $\varepsilon_n=n^{-\alpha/(2\alpha+q_0)}\log^\beta n$ for some $\beta\geq1/2$. { Note that Theorem \ref{thm:trees} is stated for $\beta=1/2$. We give a proof for the general case $\beta\geq 1/2$ under the assumptions $q_0\|f_0\|_{\Ha}\lesssim \log^{\beta}n$ and $\|f_0\|_\infty\lesssim\log^\beta n$ (Remark \ref{remark_beta}).}
 The first step requires constructing the sieve $\mathcal{F}_\mT^n\subset \mF_\mT$. For  a given $n\in\N$ and suitably large integers $q_n<k_n$ (chosen later), we define the sieve
$\mathcal{F}_\mT^n$ as consisting of step functions over small trees that split only  on a few variables, i.e.
$$
\mathcal{F}_\mT^n = \bigcup_{q = 0}^{q_n} \bigcup_{K=1}^{k_n} \bigcup_{\mS:|\mS|=q} \mF(\mV_{\mS}^K),
$$
where $\mF(\mV_{\mS}^K)$ was defined in \eqref{class:stepf}.
The optimal choice of $k_n$ and $q_n$ will follow from our considerations below.

\subsection{Condition \eqref{eq:entropy}}\label{sec:eq_entropy}

We start with a useful lemma that characterizes a useful upper bound on the covering number of the smaller sets $ \mF(\mV_{\mS}^K)$.

 \begin{lemma}\label{lemma:cover}
 Let $\mF(\mV^{K}_\mS)$ be the class of step functions  \eqref{class:stepf}. Then
 \begin{equation}\label{cover:bound}
N\Big(\tfrac{\varepsilon}{36}, \Big\{f \in \mF(\mV^{K}_\mS): \|f - f_0\|_n < \varepsilon\Big\}, \|.\|_n\Big) \leq \Delta(\mV^{K}_\mS)\left(\frac{108}{\bar{C}} \sqrt{n}\right)^K,
 \end{equation}
 where $\Delta(\mV^{K}_\mS)$ is the partitioning number of $\mV^{K}_\mS$.
 \end{lemma}
 \begin{proof}
Let $f_{\mT,\b_1}$ and $f_{\mT,\b_2} \in \mF(\mV^{K}_\mS)$ be two step functions supported on a single valid tree partition $\mT\in\mV^{K}_\mS$ with steps $\b_1\in\R^K$ and $\b_2\in\R^K$. Then, by the minimum leaf size condition, we have
$$
\frac{\bar{C}^2}{n}\|\b_1 - \b_2\|_2^2 \leq  \|f_{\mT,\b_1} -f_{\mT,\b_2}\|_n^2=\sum_{k=1}^K\mu(\Omega_k)(\beta_{1k}-\beta_{2k})^2 \leq \|\b_1 - \b_2\|_2^2.
$$ 
Denote by $f_{\mT,\wh{\b}}$ the $\|\cdot\|_n$ projection of $f_0$ onto $\mF(\mT)\subset \mF(\mV_{\mS}^K)$, the set of all step functions that live on a given partition $\mT$.  Then 
$\{\b:\|f_{\mT,\b}-f_0\|_n\leq \varepsilon\}\subset \{\b: \|\b-\wh{\b}\|_2\leq \varepsilon\sqrt{n}/\bar{C} \}$ and $\{\b:\|f_{\mT,\b}-f_0\|_n\leq \varepsilon/36\}\supset \{\b: \|\b-\wh{\b}\|_2\leq \varepsilon/36) \}$.
This relationship shows that the $\varepsilon/36$ covering number of an $||\cdot||_n$ ball  $\{\b:\|f_{\mT,\b}-f_0\|_n\leq \varepsilon\}$ can be bounded from above by the $\varepsilon/36$ covering number of an Euclidean ball of a radius $\varepsilon\sqrt{n}/\bar{C}$, which is 
bounded by {\small $\left(\frac{108}{\bar{C}} \sqrt{n}\right)^K$}. We can repeat this argument by projecting $f_0$ onto any valid tree topology $\mT\in \mV^{K}_\mS$. The number of such valid trees is no larger than $\Delta(\mV^{K}_\mS)$, which completes the proof.\qedhere
 \end{proof}

The covering number for the entire sieve $\mathcal{F}_\mT^n$  is then seen to satisfy
\begin{align*} 
&N\Big(\tfrac{\varepsilon}{36}, \Big\{f \in \mathcal{F}_\mT^n: \|f - f_0\|_n < \varepsilon\Big\}, \|.\|_n\Big)\\
&\quad\quad\quad<\sum_{q=0}^{q_n}\sum_{K=1}^{k_n}\sum_{\mS:|\mS|=q}
 N\Big(\tfrac{\varepsilon}{36}, \Big\{f \in \mF(\mV_{\mS}^K): \|f - f_0\|_n < \varepsilon\Big\}, \|.\|_n\Big).
\end{align*}
From Lemma \ref{lemma:cover} and Lemma \ref{lemma:complexity}, we  obtain the following upper bound 
 \begin{align}
&\sum_{q=0}^{q_n}\sum_{K=1}^{k_n} \binom{p}{q} (K-1)!q^K n^K\left(\frac{108}{\bar{C}} \sqrt{n}\right)^K\notag\\
&\hspace{3cm}<  \sum_{K=0}^{k_n} \left(\frac{108}{\bar{C}} q_n n^{3/2}\,K\right)^K\sum_{q=0}^{q_n} \binom{p}{q},\label{layer}
\end{align}
where we used  the fact $K!<K^K$. Next, using the regularized incomplete beta function representation of the Binomial {\sl cdf}, 
we can write
\begin{align}
\sum_{q=0}^{q_n}{p\choose q}&=2^p(p-q_n){p\choose q_n}\int_{0}^{1/2}x^{p-q_n-1}(1-x)^{q_n}\d x\notag\\
&\leq 2^{q_n+1}(p-q_n){p\choose q_n}\frac{1-1/2^{q_n+1}}{q_n+1}\leq 2^{q_n+1}{p\choose q_n}\frac{p-q_n}{q_n+1}.\label{binomial_cdf}
\end{align}
Using $\binom{p}{q_n} \leq (\e \, p/q_n)^{q_n}$, the quantity in \eqref{layer} can be bounded by
$$
\left(\frac{2\,\e\, p}{q_n}\right)^{q_n+1}\sum_{K=0}^{k_n} \left(\frac{108}{\bar{C}} q_n\, n^{3/2}\, k_n\right)^K < 
\left(\frac{2\,\e\, p}{q_n}\right)^{q_n+1}\frac{(C\,q_n\, n^{3/2}\, k_n)^{k_n+1}-1}{C\,q_n\, n^{3/2}\, k_n-1},
$$
where  $C=\left(\frac{108}{\bar{C}}\right)$. Finally, the entropy condition requires that  
the log-covering number, now upper bounded by
$$
(q_n + 1)\log(2\,\e\,p/q_n) + (k_n+1) \log(C\,q_n\, n^{3/2}\, k_n), 
$$
is no larger than (a constant multiple of) $n\,\varepsilon_n^2=n^{q_0/(2\alpha+q_0)}\log^{2\beta}n$. 
Under our assumption $\log p\lesssim n^{q_0/(2\alpha+q_0)}$,  this will be satisfied with {$q_n=\lceil C_q\mathrm{min}\{p,n^{q_0/(2\alpha+q_0)}\log^{2\beta} n/\log (p\vee n)\}\rceil$ for some $C_q\geq1$} and 
$k_n=\lfloor C_k n\,\varepsilon_n^2/\log n\rfloor \asymp n^{q_0/(2\alpha+q_0)}\log^{2\beta-1} n$.
The constants $C_q$ and $C_k$ will be determined later.

\subsection{Condition \eqref{eq:prior}}\label{sec:proof:tree1}
We wish to show that the prior assigns enough mass around the truth in the sense that
\begin{equation}
\Pi(f \in \F_\mT : \|f - f_0\|_n \leq \varepsilon_n) \geq \e^{-d\,n\varepsilon_n^2} \label{eq:prob2}
\end{equation}
for some large enough  $d>0$. 
We establish this condition by finding a lower bound on the prior probability in \eqref{eq:prob2}, using all step functions supported on a single good partition.
We denote by $\mS_0$ the true index set of active covariates, where $|\mS_0|=q_0$. According to Lemma \ref{lemma:approx}, there exists a tree-structured 
step function  $f_{\wt{\mT},\wt{\b}}\in\mF(\mV_{\mS_0}^K)$ for some $K=2^{q_0\,s}$ and $s\in\N\backslash\{0\}$
such that
\begin{equation}\label{aux1}
\|f_0 - f_{\wt{\mT},\wt{\b}}\|_n \leq   ||f_0||_{\Ha}C_1\, q_0/K^{\alpha/q_0}
\end{equation}
for some $C_1>0$.
The proof Lemma \ref{lemma:approx} is in the Supplemental Material (Section 3).

To continue with the proof of Theorem \ref{thm:trees},
We find the smallest $K=2^{s\,q_0}$ such that the  function $f_{\wt{\mT},\wt{\b}}\in\mF(\mV_{\mS_0}^K)$ in \eqref{aux1} safely approximates $f_0$ with an error that is no larger than $\varepsilon_n/2$, a constant multiple of the target rate.  Such  a $K$ will be denoted by $a_n$ and  is defined as the smallest $K$ such that  
 $\varepsilon_n \geq 2 C_0q_0 K^{-\alpha/q_0}$ for $C_0=||f_0||_{\Ha}C_1$. Then we have
\begin{equation}\label{bound}
\left(\frac{2C_0q_0}{\varepsilon_n}\right)^\frac{q_0}{\alpha} \leq a_n \leq  \left(\frac{2C_0q_0}{\varepsilon_n}\right)^\frac{q_0}{\alpha} + 1.
\end{equation}
We denote  by 
 $\wh{\mT}=\{\wh{\Omega}_k\}_{k=1}^{a_n} \in \mathcal{T}^{a_n}_{\mS_0}$ the $k$-$d$ tree partition from Lemma \ref{lemma:approx} obtained with the choice $K=a_n$. 
 The tree $\wh{\mT}$ is not only valid, but also balanced in the sense that $C^2_{min}/a_n<\mu(\widehat{\Omega}_k)\leq C_{max}^2/a_n$ for some $C_{min}<1<C_{max}$. The associated step sizes  of the approximating tree will be denoted by $\wh{\b}\in\R^{a_n}$.
Now, we lower-bound the prior probability of the neighborhood  $\{f \in \F_\mT : \|f - f_0\|_n^2 \leq \varepsilon_n^2\}$ by the prior probability of all regression trees supported on  $\wh{\mT}$ inside this neighborhood (denoted by $\mF(\wh{\mT})$):
\begin{align}
\Pi(f \in \F_\mT &: \|f - f_0\|_n^2 \leq \varepsilon_n^2)\geq  \pi(q_0)\pi(a_n) \frac{\Pi(f \in \F(\wh{\mT}) : \|f - f_0\|_n^2 \leq \varepsilon_n^2)}{\binom{p}{q_0}\Delta(\mV_{\mS_0}^{a_n})}. \label{continue}
\end{align}
For any $\b\in\R^{a_n}$, we have
$$
\|f_{\wh{\mT},\b}-f_{\wh{\mT},\wh{\b}}\|_n^2=\sum_{k=1}^{a_n}\mu(\wh{\Omega}_k)(\beta_k-\wh{\beta}_k)^2\leq \|\b-\wh{\b}\|_2^2
$$
and by  the reverse triangle inequality
$$
\| \b - \wh{\b}\|_2  \geq \|f_{\wh{\mT},\b}-f_{\wh{\mT},\wh{\b}}\|_n \geq \big| \|f_{\wh{\mT},\b} -f_0 \|_n - \|f_{\wh{\mT},\wh{\b}} - f_0\|_n \big|.
$$
Then the statement  $\|\b-\wh{\b}\|_2<\varepsilon_n/2$ implies
$
\|f_{0}-f_{\wh{\mT},\b} \|_n<\|f_{0}-f_{\wh{\mT},\wh{\b}} \|_n+\varepsilon/2<\varepsilon,
$
where the last inequality follows from the definition of $a_n$. Thus, we have
$$
\{\b:\|\b-\wh{\b}\|_2\leq \varepsilon_n/2 \}\subset \{f\in\mF(\wh{\mT}): \|f_{0}-f\|_n<\varepsilon_n\}.
$$ 
Now we can lower-bound \eqref{continue} with
\begin{equation}\label{L}
L(q_0,a_n,\mS_0,\wh{\b},\varepsilon_n)\equiv\pi(q_0)\pi(a_n) \frac{\Pi(\b\in\R^{a_n}:\|\b-\wh{\b}\|_2\leq \varepsilon_n/2)}{\binom{p}{q_0}\Delta(\mV_{\mS_0}^{a_n})}.
\end{equation}
In order to bound $\Pi(\b \in \R^{a_n} : \|\b - \wh{\b}\|_2 \leq \varepsilon_n/2)$, we follow the computations
 of \cite{ghosal_vdv}, Theorem 12, to obtain
\begin{align}\label{vdv_lower}
\Pi\Big(\b\in \mathbb{R}^{a_n} &: \|\b - \wh{\b}\|_2 \leq  \varepsilon_n/2\Big)
\geq\frac{ 2^{-a_n}   \e^{-\|\wh{\b}\|_2^2-\varepsilon_n^2/8}}{\Gamma(\frac{a_n}{2})\frac{a_n}{2}} \left(\frac{\varepsilon_n^2}{4}\right)^\frac{a_n}{2}.
\end{align}
From the triangle inequality (and because $\wh{\mT}$ is balanced) we have
$$
\|\wh{\b}\|_2 \leq  \frac{\sqrt{a_n}}{C_{min}}\|f_{\wh{\mT},\wh{\b}} \|_n\leq   \frac{\sqrt{a_n}}{C_{min}}\left(\|f_{\wh{\mT},\wh{\b}}-f_0\|_n+\|f_0\|_\infty\right)\leq 
  \frac{\sqrt{a_n}}{C_{min}}\left(\frac{\varepsilon_n}{2}+\|f_0\|_\infty\right).
$$ 
Thereby we can write $\|\wh{\b}\|_2^2 \leq C_2\|f_0\|_\infty^2 a_n$ for some constant $C_2>0$. Now we continue with a  lower bound to $L(q_0,a_n,\mS_0,\wh{\b},\varepsilon_n)$ defined in \eqref{L}.
Using the following facts $\Gamma(x) \leq x^{x}$, $\binom{p}{q_0} \leq (\e p/q_0)^{q_0}$ and $\Delta(\mV_{\mS_0}^{a_n}) <(a_n\,q_0\,n)^{a_n}$ (Lemma \ref{lemma:complexity}) and using \eqref{vdv_lower}, we arrive at the following lower bound:
\begin{equation}\label{aux2}
\frac{\pi(a_n) c^{-q_0}p^{-aq_0}}{ \left(\frac{\e\, p}{q_0}\right)^{q_0} (a_n\,q_0\,n)^{a_n} } \e^{-\varepsilon_n^2/8-a_n(C_2\|f_0\|_\infty^2 +\log2)} \left(\frac{\varepsilon_n^2}{4}\right)^\frac{a_n}{2}   \left(\frac{2}{a_n}\right)^{a_n/2 + 1}.
\end{equation}
Condition \ref{eq:prior} will be satisfied if  this quantity is at least as large as $\e^{-d\,n\varepsilon_n^2}$ for some large $d>0$. We  denote  $c(p, q_0, a_n) = c^{-q_0}p^{-aq_0}(q_0/\e p)^{q_0}$. Then we can rewrite \eqref{aux2} as
$$ \pi(a_n)c(p, q_0, a_n) \left(\sqrt{2}\e^{C_2\|f_0\|_\infty^2+ \log 2}a_n^{3/2}n\right)^{-a_n} \frac{2}{a_n} \left(\frac{\varepsilon_n}{q_0}\right)^{a_n}\e^{-\varepsilon_n^2/8}. $$
Taking minus the log of this quantity, Condition \eqref{eq:prior} will be met when
\begin{equation}\label{terms}
-\log c(p, q_0, a_n) -\log{\pi(a_n)} + a_n \log\left(\sqrt{2}\e^{C_2\|f_0\|_\infty^2 +\log2}a_n^{3/2}n\right) + a_n\log\left(\frac{q_0}{\varepsilon_n}\right)
\end{equation}
is smaller than a constant multiple of $n\varepsilon_n^2$. Above, we omitted the small terms ${\varepsilon_n^2/8}$ (since $\varepsilon_n \to 0$) and $\log(a_n/2)$. 
First, we note that
$$
-\log c(p, q_0, a_n)\leq q_0\log (c\,p^{a+1}\e/q_0).
$$ 
With $q_0\lesssim \log^\beta n$ and  $\log p\lesssim n^{q_0/(2\alpha+q_0)}$, we obtain $-\log c(p, q_0, a_n)\lesssim n\varepsilon_n^2$.
Next, focusing on the last term in \eqref{terms}, we obtain (from the left inequality in \eqref{bound}) the following bound 
 $$
 q_0/\varepsilon_n\leq a_n^{\alpha/q_0}/(2C_0)
 $$ 
 and hence $q_0/\varepsilon_n\lesssim a_n$ for  $\alpha\leq 1$ and $q_0\geq 1$.
Moreover, from the right inequality in \eqref{bound} we obtain for $\varepsilon_n=n^{-\alpha/(2\alpha+q_0)}\log^\beta n$ and $2C_0q_0\lesssim \log^\beta n$
\begin{equation}\label{aux3}
a_n\lesssim n^{q_0/(2\alpha+q_0)}.
\end{equation}
  Under our assumption $\|f_0\|_\infty\lesssim \log^\beta n$, \eqref{aux3} immediately yields $a_n\|f_0\|_\infty^2\lesssim n\varepsilon_n^2$. All of these considerations, combined with the fact $-\log{\pi(a_n)}\lesssim  a_n\log a_n$, yield  the following  leading term behind the last three summands in \eqref{terms}: $a_n\log (a_n^{3/2}n)$.  Using \eqref{aux3}, we obtain $a_n\log (a_n^{3/2}n)\lesssim a_n \log n\lesssim n\varepsilon_n^2$ for $\beta\geq1/2$.  Altogether, there exists $d>0$ such that \eqref{eq:prob2} is satisfied.

\subsection{Condition \eqref{eq:remain}}\label{sec:proof:tree:eq:remain}


In order to establish  Condition  \ref{eq:remain}, we begin by noting
$\Pi(\mF_\mT\backslash\mF_\mT^n)< \Pi(q>q_n)+ \Pi(K>k_n).
$
Thus, the condition will be met when both
$\Pi(q>q_n)=o(\e^{-(d+2)\,n\varepsilon_n^2})$ and $\Pi(K>k_n)=o(\e^{-(d+2)\,n\varepsilon_n^2})$, where $d$ is the constant deployed in Section \ref{sec:proof:tree1}.
First, we find that
$$
\Pi(q>q_n)\lesssim\sum_{k=q_n+1}^p(c\,p^a)^{-k}=(c\,p^a)^{-(q_n+1)}\frac{1-(cp^a)^{-(p-q_n)}}{1-1/(cp^a)}<(c\,p^a)^{-q_n}.
$$
With our choice  $q_n=\lceil C_q\mathrm{min}\{p,n^{q_0/(2\alpha+q_0)}\log^{2\beta} n/\log (p\vee n)\}\rceil$, it turns out that
\begin{equation}\label{eq:qdecay}
\Pi(q>q_n)\e^{(d+2)\,n\varepsilon_n^2}< \e^{-q_n[\log c+a\log p]+(d+2)\,n\varepsilon_n^2}\rightarrow 0
\end{equation}
for a large enough constant $C_q>0$. { Indeed, for $p>n$ we have $q_n\log p\asymp n\,\varepsilon_n^2$. For $p<n\,\varepsilon_n^2/\log n$ we have $q_n=p$ and $\Pi(q>q_n)=0$. Finally, for $n\,\varepsilon_n^2/\log n\leq p\leq n$, we have $q_n\log n\asymp n\,\varepsilon_n^2$ and $\log p\geq \frac{q_0}{2\alpha+q_0}\log n+(2\beta-1)\log\log n$.
For $c>1$ and $\beta\geq 1/2$,  we can write $q_n[\log c+a\log p]\leq C_q\frac{aq_0}{2\alpha+q_0}n^{q_0/(2\alpha+q_0)}\log^{2\beta} n$ and \eqref{eq:qdecay} holds for $C_q$ large enough.} Next, we apply the Chernoff bound for $\Pi(K>k_n)$. Namely, for any $t>0$ we can write
\begin{equation}\label{chernoff}
\Pi(K>k_n)<\e^{-t\,(k_n+1)}\mathbb{E}\, \e^{t\,K}\propto \e^{-t\,(k_n+1)}\sum_{k=1}^\infty \frac{\left(\e^{t}\lambda\right)^k}{k!}
\propto  \e^{-t\,(k_n+1)}\left(\e^{\e^t\lambda}-1\right).
\end{equation}
With our choice $k_n=\lfloor C_kn\varepsilon_n^2/\log n\rfloor\asymp C_kn^{q_0/(2\alpha+q_0)}\log^{2\beta-1}n$  (Section \ref{sec:eq_entropy}) and  with $t=\log k_n$ we obtain
$$
\Pi(K>k_n)\e^{(d+2)\,n\varepsilon_n^2}\lesssim \e^{-(k_n+1)\log k_n+ \lambda k_n+(d+2)\,n\varepsilon_n^2}\rightarrow 0
$$
for a large enough constant $C_k$.

\section{Proof of Theorem \ref{thm:additivetrees}}\label{sec:proof:additivetrees}
We aim to establish conditions  \eqref{eq:entropy}, \eqref{eq:prior} and  \eqref{eq:remain} for $\varepsilon_n^2=\sum_{t=1}^{T_0}({\varepsilon}_n^t)^2$, where ${\varepsilon}_n^t=n^{-\alpha^t/(2\alpha^t+q_0^t)}\log^{\beta^t} n$. 
Our sieve consists of valid forests with either (a) many trees that are small (weak learners), or (b) a few large trees (strong learners). We impose a joint requirement on $\sum_{t=1}^TK^t$ so that the overall number of leaves in the ensemble is small. At the same time, we require that  $\sum_{t=1}^Tq^t$ (the upper bound on the number of active variables in the ensemble) is small as well. The sieve is constructed as follows:
\begin{equation}\label{sieve2}
\mF_\mE^n=\bigcup_{T=1}^\infty\bigcup\limits_{\substack{\q: \sum_{t=1} ^Tq^t\leq s_n}} \bigcup\limits_{\substack{\bmS: |\mS^t|=q^t}}
\bigcup\limits_{\substack{\bm{K}:\sum_{t=1}^T K^t\leq z_n}} \mF(\mV\mE_{\bmS}^{\bm K})
\end{equation}
for some integer values $s_n$ and $z_n$.
Throughout this section we denote $\bar{K}=\frac{1}{T}\sum_{t=1}^T K^t$.
\subsection{Condition \ref{eq:entropy}}\label{sec:sieve:ensemble}
{We first obtain the following upper bound on the log-covering number 
\begin{equation}\label{continue1}
\log N\left(\frac{\varepsilon}{36},\{f_{\mE,\mB}\in\mF(\mE):\|f_{\mE,\mB}-f_0\|_n<\varepsilon\},\|\cdot\|_n\right)\lesssim 
(T\times \bar K)\log(108\sqrt{n})
\end{equation}
where
 $\mF(\mE)=\left\{f_{\mE,\mB}:[0,1]^p\rightarrow\R:f_{\mE,\mB}(\x)=\sum_{t=1}^Tf_{\mT^t,\b^t}(\x);\b^t\in\R^{K^t} \right\}$ is the set of all additive step functions supported on a single partition ensemble $\mE\in\mV\mE_{\bmS}^{\bm{K}}$. We denote by $\wt{\mT}(\mE)=\{\widetilde{\Omega}_k\}_{k=1}^{K(\mE)}$ the global partition associated with $\mE$, consisting of $K(\mE)$ global cells. For $\mB_1,\mB_2\in\R^{T\times \bar{K}}$, we denote by $f_{\mE,\mB_1},f_{\mE,\mB_2}\in\mF(\mE)$ two additive regression trees that  sit on the same partition ensemble $\mE$. Let $\bar{\b}_1=\A(\mE)\mB_1$ and $\bar{\b}_2=\A(\mE)\mB_2$ be the aggregated step sizes, as defined in \eqref{bar_mapping}, where $\A(\mE)$ is the stretching matrix.  Then we can write
$$
\frac{1}{n}||\bar{\b}_1-\bar{\b}_2||_2^2 \leq ||f_{\mE,\mB_1}-f_{\mE,\mB_2}||_n^2=\sum_{k=1}^{K(\mE)}\mu(\widetilde{\Omega}_k)(\bar{\beta}_{1k}-\bar{\beta}_{2k})^2\leq ||\bar{\b}_1-\bar{\b}_2||_2^2.
$$
Deploying the singular value decomposition $\A(\mE)=\U\DD\V^T$
we write $\wt\mB_1=\bm V^T\mB_1\in\R^{\wt K}$ and $\wt\mB_2=\bm V^T\mB_2\in\R^{\wt K}$, where $\wt K\leq\min\{K(\mE),T\bar K\}$.  Using the fact that $\U$ is unitary, we have
$$
\frac{1}{n}||\DD(\tilde\mB_1-\tilde\mB_2)||_2^2\leq||f_{\mE,\mB_1}-f_{\mE,\mB_2}||_n^2\leq ||\DD(\tilde\mB_1-\tilde\mB_2)||_2^2.
$$
We write $f_{\mE,\mB_2}$ to be the projection of $f_0$ onto $\mF(\mE)$ and note that 
$\{\mB\in\R^{T\bar K}: \|f_{\mE,\mB}-f_{\mE,\mB_2}\|_n\leq \varepsilon\}\subset \{\wt\mB:\|\bm D(\wt\mB-\wt\mB_2)\|_2\leq \sqrt{n} \varepsilon \}$
and 
$\{\mB\in\R^{T\bar K}: \|f_{\mE,\mB}-f_{\mE,\mB_2}\|_n\leq \varepsilon/36\}\supset \{\wt\mB:\|\bm D(\wt\mB-\wt\mB_2)\|_2\leq  \varepsilon/36 \}.$
The covering number of $\{f_{\mE,\mB}\in\mF(\mE):\|f_{\mE,\mB}-f_0\|_n\leq \varepsilon\}$ can be thus bounded from above by the 
minimal number of   smaller ellipsoids $\{\wt\mB_1\in\R^{\wt K}:\|\bm D(\wt\mB_1-\wt\mB_2)\|_2\leq  \varepsilon/36 \}$
needed to cover a larger ellipsoid $\{\wt\mB_1\in\R^{\wt K}:\|\bm D(\wt\mB_1-\wt\mB_2)\|_2\leq \sqrt{n} \varepsilon \}$.
Since these ellipsoids have {\sl the same scaling factors} $\bm D$, this number is {\sl the same} as the minimal number of 
 little balls $ \{\wt\mB_1\in\R^{\wt K}:\|\wt\mB_1-\wt\mB_2\|_2\leq  \varepsilon/36 \}$ needed to  cover $ \{\wt\mB_1\in\R^{\wt K}:\|\wt\mB_1-\wt\mB_2\|_2\leq \sqrt{n} \varepsilon \}$. 
 The links between coverings of ellipsoids and balls can be found, for instance, in Dumer \cite{dumer}.
This altogether implies that  the covering number is  bounded by $(108\sqrt{n})^{T\bar K}$.
}

Now we find an upper bound on the number of valid ensembles $\mE\in\mV\mE_{\bmS}^{\bm K}$  inside the sieve $\mF_\mE^n$. 
To start, we note that given $(T,\q,\bm{K},\bmS)$, there are at most
$\prod_{t=1}^T(K^t\,q^t\,n)^{K^t}$ valid
ensembles $\mV\mE_{\bmS}^{\bm K}$. This bound is obtained from Lemma \ref{lemma:complexity} by combining all possible $T$-tuples of trees.\footnote{The order of trees in $\mE$ matters.}  Given $(T,\q)$, there are  $\prod_{t=1}^T{p\choose q^t}$ sets of subsets $\bm{\mS}=\{\mS^1,\dots,\mS^T\}$ satisfying the constraint $|\mS^t|=q^t$. This leads to an overall upper bound
\begin{align*}
\sum_{T=1}^{\min\{ s_n, z_n\}}&\sum_{\bm{K}:\sum_{t=1}^T K^t\leq z_n} \sum_{\q:\sum_{t=1}^T q^t\leq s_n}\prod_{t=1}^T{p\choose q^t}\left(K^t\,q^t\,n\right)^{K^t}\notag \\
&< \sum_{T=1}^{\min\{s_n, z_n\}}\sum_{\bm{K}:\sum_{t=1}^T K^t\leq z_n} \left(z_n\,s_n\,n\right)^{z_n} \sum_{\q:\sum_{t=1}^T q^t\leq s_n}\prod_{t=1}^T\left(\frac{p\,\e}{q^t}\right)^{q^t}\\
&<s_n^{s_n+1}z_n^{z_n+1} \left(z_n\,s_n\,n\right)^{z_n}\left({p\,\e}\right)^{s_n}.
\end{align*}
Combining this bound with \eqref{continue1},
 we obtain the following  bound
\begin{align}
&\log N\Big(\tfrac{\varepsilon}{36}, \Big\{f \in \mathcal{F}_\mE^n: \|f - f_0\|_n < \varepsilon\Big\}, \|.\|_n\Big)<(s_n+1)\log s_n+(z_n+1)\log z_n\notag  \\
&\quad\quad\quad\quad+ z_n\log\left(z_n\,s_n\,n\right)+s_n\log\left({p\,\e}\right)+z_n\log\left(108\, \sqrt{n}\right). \label{overall_bound}
\end{align}
Condition \ref{eq:entropy} will be met when \eqref{overall_bound} is smaller than (a constant multiple of) $n\varepsilon_n^2=\sum_{t=1}^{T_0}n(\varepsilon_n^t)^2$. 
With the choice $z_n=\lfloor C_z n\varepsilon_n^2/\log n\rfloor$ and {$s_n= \lceil C_s n^{q_0/(2\alpha+q_0)}\log^{2\beta} n/\log (p\vee n)\rceil$}, where $C_s$ and $C_z$ are  large enough constants to be determined later, this condition is satisfied.

\subsection{Condition \eqref{eq:prior}}\label{sec:eq:prior:ens}
To establish Condition \eqref{eq:prior} for tree ensembles, we begin by finding a single additive tree that approximates well. We will heavily leverage our findings from Section \ref{sec:proof:tree1},  noting that the problem of approximating an additive function $f_0$ with a sum of trees can be decomposed into smaller problems of approximating each layer $f_0^t$ separately.

Denote by $a_n^t$ the smallest leaf size of a $k$-$d$ tree (defined in Remark \ref{remark:kd})  needed to approximate $f_0^t$ with an error smaller than  $\varepsilon_n^t/2$, where $\varepsilon_n^t=n^{-\alpha^t/{(2\alpha^t+q^t_0)}}\log^{\beta^t}n$.  Such a tree step function approximation exists according to Lemma \ref{lemma:approx} when $\mX$ is $(M,\mS_0^t)$-regular. We will denote this approximation with $f_{\wh{\mT}^t,{\wh{\b}}^t}(\x)$. Moreover, with $\bm{a}_n=(a_n^1,\dots,a_n^{T_0})'$ we denote the vector of such minimal tree sizes, where each  $a_n^t$ satisfies \eqref{bound} with $q_0^t,\alpha^t$ and $\varepsilon_n^t$. Next, we will denote by $\wh{\mE}=\{\wh{\mT}^1,\dots,\wh{\mT}^{T_0}\}$ the approximating partition ensemble with step heights $\wh{\mB}=(\wh{\b}^{1\prime},\dots,\wh{\b}^{T_0\prime})'$. 
The individual tree approximations $f_{\wh{\mT}^t,\wh{\b}^t}(\x)$ are woven into an approximating forest $f_{\wh{\mE},\wh{\mB}}(\x)=\sum_{t=1}^{T_0}f_{\wh{\mT}^t,\wh{\b}^t}(\x)\in \F(\mV\mE_{\bmS_0}^{\bm{a}_n})$,
where $\bm{\mS}_0=\{\mS^1_0,\dots,\mS^{T_0}_0\}$.

Arguing as in Section \ref{sec:proof:tree1},  the statement
$\| \wh{\b}^t-\b^t\|_2<\frac{\varepsilon_n^t}{2}$ for all $1\leq t\leq T_0$ implies  $\|f^t_0- f_{\wh{\mT}^t,\b^t}\|_n<\varepsilon_n^t$ for all $1\leq t\leq T_0$ and $\b^t\in\R^{a_n^t}$. This further implies
$$
\|f_0-f_{\wh{\mE},\mB}\|_n \leq  \sum_{t=1}^{T_0}\|f_0^t-f_{\wh{\mT}^t, \b^t}\|_n\leq \sum_{t=1}^{T_0}\varepsilon_n^t \leq \sqrt{T_0} \varepsilon_n,
$$
 for any $\mB=(\b^{1\prime},\dots,\b^{T_0\prime})'\in\R^{T_0\times \bar{\bm{a}}_n}$, where the final inequality is due to Cauchy-Schwarz and where $\bar{\bm{a}}_n=\frac{1}{T_0}\sum_{t=1}^{T_0}a_n^t$.  Denote by $\mF(\wh{\mE})$ the set of all additive trees supported on the partition ensemble $\wh{\mE}$.
Then we can write
\begin{align}
\Pi&\left(f\in\mF(\wh{\mE}):  \| f_0-f\|_n\leq \varepsilon_n\right) \notag\\
&\quad\quad \geq\Pi\left(\b\in \R^{\sum_{t=1}^{T_0}a_n^t}: \|\wh{\b}^t-\b^t \|_2 \leq \frac{\varepsilon_n^t}{2\sqrt{T_0}}\quad\text{for each }\,t=1,\dots, T_0\right)\notag\\
&\quad\quad = \prod_{t=1}^{T_0}
\Pi\left(\b^t\in \R^{a_n^t}: \|\wh{\b}^t-\b^t \|_2 \leq \frac{\varepsilon_n^t}{2\sqrt{T_0}}\right).\notag \label{step1} 
\end{align}
 Because we assumed $\beta_{j}^t\sim\mathcal{N}(0,1/T)$, given $T$, we  can directly use  \eqref{vdv_lower} to lower-bound the above with
\begin{equation}\label{step2}
 \prod_{t=1}^{T_0}\frac{ 2^{-a_n^t}   \e^{-\|\wh{\b}^t\|_2^2-(\varepsilon_n^t)^2/8}}{\Gamma(\frac{a_n^t}{2})\frac{a_n^t}{2}} \left(\frac{(\varepsilon_n^t)^2}{4}\right)^\frac{a_n^t}{2}.
\end{equation}
Because each tree $\wh{\mT}^t$ is a $k$-$d$ tree and is by definition balanced, we have $\|\wh{\b}^t\|_n^2\lesssim a_n^t\|f_0^t\|_\infty^2$.
Now we can directly apply all our calculations from Section \ref{sec:proof:tree1}. In particular, using \eqref{step2} and noting that $\Delta(\mV\mE_{\bmS_0}^{\bm{a}_n})<\prod_{t=1}^{T_0}\Delta(\mV_{\mS_0^t}^{a_n^t})$,
we obtain
\begin{align*}
\Pi\left(f\in\mF_\mE:  \| f_0-f\|_n\leq \varepsilon_n\right)&\geq \pi(T_0)\pi(\bm{q}_0\C T_0)\pi(\bm{a}_n\C T_0)\pi(\bmS_0\C T,\bm{q}_0)\times\\
&\times\pi(\wh{\mE}\C \bmS_0,\bm{a}_n)\Pi\left(f\in\mF(\wh{\mE}):  || f_0-f_{\wh{\mE},\mB}||_n\leq \varepsilon_n\right)\\
&> \pi(T_0)\prod_{t=1}^{T_0}L(q_0^t,a_n^t,\mS_0^t,\wh{\b}^t,\varepsilon_n^t),
\end{align*}
where $L(\cdot)$ was defined in \eqref{L}. It follows from Section \eqref{sec:proof:tree1} that 
$$
-\log L(q_0^t,a_n^t,\mS_0^t,\wh{\b}^t,\varepsilon_n^t)\lesssim n(\varepsilon_n^t)^2
$$
for each $1\leq t\leq T_0$ when $q_0^t\lesssim \log^{\beta^t}n$, $\log p\lesssim \min\limits_{1\leq t\leq T_0}n^{q^t/(2\alpha^t+q_0^t)}$ and $T_0\lesssim n$.  The last condition $T_0\lesssim n$ is needed under our prior $K^t\sim \mathrm{Poisson}(\lambda/T)$ so that $-\log \pi(a_n)\lesssim a_n\log a_n+a_n\log T_0\lesssim n(\varepsilon_n^t)^2=n^{q_0^t/(2\alpha^t+q_0^t)}\log^{2\beta^t} n$ for $\beta^t\geq 1/2$.
Putting all the pieces together, we obtain the following lower bound
$$
\Pi\left(f\in\mF_{\mE}:  \| f_0-f\|_n\leq \varepsilon_n\right)\geq \pi(T_0)\e^{-d\,n \sum_{t=1}^{T_0}(\varepsilon_n^{t})^2}
$$
for some suitably large $d>0$. The last requirement needed for Condition \eqref{eq:prior} to be satisfied  is  that $\pi(T_0)\geq \e^{-d\,n\varepsilon_n^2}$.
Our prior $\pi(T) \propto \e^{-C_T T}$ safely satisfies this requirement.

\subsection{Condition \eqref{eq:remain}}\label{sec:remain:ens}
The condition entails showing  that  $\Pi(\mF_\mE\backslash\mF_\mE^n)=o(\e^{-(d+2)\,n\varepsilon_n^2})$ for $d$ deployed in the previous section. It suffices to show that 
 $$\left[\Pi\left((T, \bm{K}) : \sum_{t=1}^T K^t>z_n\right)+\Pi\left( (T, \bm{q}): 
\sum_{t=1}^T q^t>s_n\right)\right]\e^{(d+2)\,n\varepsilon_n^2}\rightarrow 0.$$
Because we assume
$ K^t \mid T \iid \mathrm{Poisson}(\lambda/T)$ for some $\lambda\in\R$ (according to our definition in \eqref{prior:KK}), 
we can  apply a similar Chernoff bound  as in \eqref{chernoff}. Namely, we have for any $\gamma>0$
\begin{align*}
\Pi\left( \sum_{t=1}^T K^t>  z_n\right)&=\sum_{T=1}^{\infty}\pi(T)\Pi\left( \sum_{t=1}^T K^t> z_n \mid  T \right)\\
&\lesssim \e^{-\gamma(z_n+1)}\sum_{T=1}^\infty\pi(T)\left(\e^{\e^\gamma\lambda/T}-1\right)^T\lesssim  \e^{-\gamma(z_n+1)+\e^\gamma\lambda}.
\end{align*}
With $z_n=\lfloor C_z n\varepsilon_n^2/\log n\rfloor\sim \sum_{t=1}^{T_0}n^{q_0^t/(2\alpha^t+q_0^t)}\log^{2\beta^t-1}n$ and $\gamma=\log z_n$, we have $\Pi\left( \sum_{t=1}^T K^t> z_n\right)\e^{d\,n\varepsilon_n^2}\rightarrow 0$ for a large enough  constant $C_z>0$. 
Next,  with the independent product prior \eqref{prior:dim2},  the Chernoff bound  gives
\begin{align*}
\Pi\left(\sum_{t=1}^T q^t > s_n \mid T\right) \leq \e^{-\gamma (s_n+1)} \prod_{t=1}^T \mathbb{E}\left[\e^{\gamma q^t}\right] \lesssim \e^{-\gamma (s_n+1)}\e^{-T\log\left[1-\e^\gamma/(cp^a)\right]}
\end{align*}
for any $\gamma>0$, where we used the fact 
$$
 \mathbb{E}\left[\e^{\gamma q^t}\right] =\sum_{q=0}^p \left[\e^{\gamma\,}/(cp^a)\right]^q<\frac{1}{1-\e^{\gamma\,}/(cp^a)}.
$$
{With $\gamma=\log (p\vee n)$ and $a>2$, we can write
\begin{align*}
\Pi\left(\sum_{t=1}^T q^t > s_n \mid T\right)&\lesssim \e^{-(s_n+1)\,\log (p\vee n)}\e^{-T\log\left[1-1/(cp^{a-1})\right]}\\
&< \e^{-(s_n+1)\log (p\vee n)-T\log\left[1-1/(cp)\right]}.
\end{align*}
Next,  we have 
$$
\Pi\left(\sum_{t=1}^T q^t > s_n\right)\lesssim \e^{ -(s_n+1)\log{(p\vee n)}}\sum_{T=1}^\infty \pi(T) \e^{T\log[1+1/(cp-1)]}.
$$
With $\pi(T)\propto \e^{-C_T T}$, where $C_T>\log 2$,  we have
$$
\Pi\left(\sum_{t=1}^T q^t > s_n\right)\lesssim \e^{ - (s_n+1)\log{(p\vee n)}}\sum_{T=1}^\infty \e^{-T(C_T-\log 2)}\lesssim \e^{ - (s_n+1)\log{(p\vee n)}}.
$$
With  $s_n = \lfloor C_s n\varepsilon_n^2 / \log{(p\vee n)}\rfloor $ we have $\Pi\left(\sum_{t=1}^T q^t >s_n\right)\e^{(d+2)\,n\varepsilon_n^2}\rightarrow 0$ for a large enough constant $C_q$.
}

\section{Proof of Theorem \ref{thm:weak_learner}}  \label{proof:thm:weak_learner}
The sieve will be very similar to \eqref{sieve2}. The only difference is that each tree in the ensemble is now constrained to depend on the same set of active variables $\mS$. To mark this difference, we have denoted the partition ensembles with $\mV\mE_{\mS}^{\bm K}$ instead of $\mV\mE_{\bmS}^{\bm K}$. 
Throughout this section, we use the following sieve:
\begin{equation}\label{sieve3}
\mF_\mE^n=\bigcup_{T=1}^\infty\bigcup_{q=0}^{q_n} \bigcup\limits_{\substack{\mS: |\mS|=q}}
\bigcup\limits_{\substack{\bm{K} : \sum_{t=1}^T K^t \leq z_n}} \mF(\mV\mE_{\mS}^{\bm K}).
\end{equation}

\subsection{Condition \ref{eq:entropy}}
Our sieve \eqref{sieve3} is embedded in \eqref{sieve2}, where the number of ensembles $\mE$ inside $\mF_\mE^n$ is now upper-bounded by
\begin{align*}
\sum_{T=1}^{z_n}& \sum_{q=0}^{q_n}\sum_{\bm{K}:\sum_{t=1}^T K^t\leq z_n}{p\choose q}\prod_{t=1}^T\left(K^t\,q\,n\right)^{K^t}<z_n^{z_n+1}(q_n+1) \left(z_n\,q_n\,n\right)^{z_n}\left({p\,\e}\right)^{q_n}.
\end{align*}
Using the same arguments as in Section \ref{sec:sieve:ensemble}, we choose  $z_n =\lfloor C_zn\varepsilon_n^2/\log{n}\rfloor$ and {$q_n =\lceil C_q\mathrm{min}\{p,n^{q_0/(2\alpha+q_0)}\log^{2\beta} n/\log (p\vee n)\}\rceil$}, for which the condition holds.

\subsection{Condition \eqref{eq:prior}}
The key ingredient for establishing  Condition \eqref{eq:prior} is the following lemma on the existence of a tree ensemble that approximates $f_0$ well.

\begin{lemma}\label{lemma:approx:ensemble}
Assume $f \in \Ha_{p}\,\cap\,\mC(\mS)$, where $|\mS|=q$, and that $\mX$ is $(M,\mS)$-regular. Then  for any $s\in\N\backslash\{0\}$, there exists an additive tree  function $f_{\wh{\mE},\,\wh{\mB}}\in \mF(\mV\mE_{\mS}^{\bm{K}})$  consisting of $T=2^{s\,q-1}$ trees, each with $K^t=s\,q+1$ leaves, such that 
$$
\|f - f_{\wh{\mE},\,\wh{\mB}}\|_n \leq   ||f||_{\Ha}C\, M^\alpha\, q/K(\wh{\mE})^{\alpha/q}
$$ 
for some $C>0$, where $K(\wh{\mE})=2^{s\,q}$. 
\end{lemma}
\begin{proof}
From Lemma \ref{lemma:approx}, we know that there exists a single tree function $f_{\wh{\mT},\wh{\b}}\in\mF(\mV_\mS^{K})$ with $\wh{K}=2^{s\,q}$ leaves which approximates well.
We regard  the full symmetric tree $\wh{\mT}$ as the global partition  of the approximating partition ensemble $\wh{\mE}$, i.e. $\wt{\mT}(\wh{\mE})=\wh{\mT}$. Moreover,  $\wh{\b}$ is regarded as the vector of aggregated steps, i.e. $\bar{\b}=\wh{\b}$ (the definition of the aggregated steps is in \eqref{bar_mapping}). The actual ensemble $\wh{\mE}$ is obtained from $\wh{\mT}$ by redistributing the cuts among $T$ trees, each  with $K^t=K$ leaves, in the following way. We take completely imbalanced trees that keep refining one cell until the resolution reaches the tree depth $\log_2\wh{K}$. These trees have $K^t=\log_2\wh{K}+1$ leaves and we need $T=\wh{K}/2$ of those to sum up towards $\wh{\mT}$. This decomposition is illustrated in Figure \ref{bayestrees}, where a full symmetric  tree $\wh{\mT}$ (Figure \ref{bayestrees2}) with $\wh{K}=8$ leaves has been trimmed into ${T}=\wh{K}/2=4$ smaller trees with $K^t=\log_2 \wh{K}+1=4$ leaves.
The decomposition yields a tree ensemble $\wh{\mE}=\{\mT^1,\dots,\mT^T\}$ with a stretching matrix $\A(\wh{\mE})=[\mathrm{I}_{\wh{K}},\bm{A}_1]$ (after a suitable permutation of columns), where $\bm{A}_1$ is some binary matrix.  It follows from Lemma 1(g) of Govaerts and Pryce \cite{singul} that  $\lambda^2_{min}(\wh{\mE})=1+\sigma^2_{min}(\A_1)\geq 1$, where $\sigma_{min}(\A_1)$ denotes the smallest singular value of $\A_1$. Moreover, we have $K(\wh{\mE})=\wh{K}$. Finally, we use \eqref{bar_mapping} to obtain the individual tree steps  via $\wh{\mB}=(\A(\wh{\mE})'\A(\wh{\mE}))^{\dagger}\A(\wh{\mE})'\bar{\b},$ where $\A^\dagger$ denotes the Moore pseudoinverse of $\A$. The rest follows from Lemma \ref{lemma:approx}. \qedhere

\end{proof}

Now we proceed with Condition \ref{eq:prior}.
Denote by $\wh{\mE}$ the approximating ensemble from Lemma \ref{lemma:approx:ensemble}. Recall that the global partition $\widetilde{\mT}(\wh{\mE})=\{\widetilde{\Omega}_k\}_{k=1}^{K(\wh{\mE})}$ is a $k$-$d$ tree, which is balanced in the sense that $C^2_{min}/ K(\wh{\mE})\leq \mu(\widetilde{\Omega}_k)\leq C_{max}^2/ K(\wh{\mE})$ for some constants $C_{min}<1<C_{max}$ and $k=1,\dots, K(\wh{\mE})$. Next, we find  the smallest  $K(\wh{\mE})$  such that $ ||f||_{\Ha}C\, M^\alpha\, q_0/K(\wh{\mE})^{\alpha/q_0}<\varepsilon_n/2$.  This value will be denoted by  $a_n$ and it
 satisfies \eqref{bound}. Next, we denote by $\wh{T}=a_n/2$ the number of approximating trees and by  $\wh{\bm{K}}=(\wh{K}^1,\dots,\wh{K}^T)'$ the vector of leaves, where $\wh{K}^t=\log_2 a_n +1$ (again we are using the construction from Lemma \ref{lemma:approx:ensemble}). Then, using similar arguments as in Section \ref{sec:eq:prior:ens} we can lower-bound $\Pi(f\in\mF_\mE : \|f - f_0\|_n \leq \varepsilon_n) \label{eq:prob}
$ with
\begin{equation}\label{lb}
\frac{\pi(\wh{T})\pi(\wh{\bm{K}}\C \wh{T}) \pi(q_0) }{ \left(\frac{\e\, p}{q_0}\right)^{q_0}\prod_{t=1}^{\wh{T}}(\wh{K}^t\,q_0\,n)^{\wh{K}^t}} \Pi\left(f\in\mF(\wh{\mE}): ||f - f_0||_n\leq \varepsilon_n\right),
\end{equation}
where $\mF(\wh{\mE})$ consists of all additive tree functions supported on $\wh{\mE}$.
We denote by  $\widetilde{a}_n=\sum_{t=1}^{\wh{T}}\wh{K}^t=a_n/2(\log_2 a_n+1)$ and by $\wh{\mB}\in\R^{\wt{a}_n}$ the steps of the approximating additive trees from Lemma \ref{lemma:approx:ensemble}.  Because  $\mu(\widetilde{\Omega}_k)\leq C_{max}^2/K(\wh{\mE})$ we obtain  for any arbitrary vector $\mB\in\R^{\wt{a}_n}$ (similarly as in Section \ref{sec:sieve:ensemble})
$$
\|f_{\wh{\mE},\mB}-f_{\wh{\mE},\wh{\mB}}\|_n\leq C_{max}\lambda_{max}(\wh{\mE})/\sqrt{K(\wh{\mE})}\|\mB-\wh{\mB}\|_2
$$
and thereby
\begin{align*}
\|\mB-\wh{\mB}\|_2&\geq  \sqrt{K(\wh{\mE})}/(C_{max}\lambda_{max}(\wh{\mE}))\big| \|f_0-f_{\wh{\mE},\mB}\|_n-\|f_0-f_{\wh{\mE},\wh{\mB}}\|_n\big|.
\end{align*}
Combined with the fact $\lambda_{\max}^2(\mE)\leq K(\mE)\widetilde{a}_n$ (as shown in the proof of Lemma \ref{lem:eigenvalupper}), the statement
$
\|\mB-\wh{\mB}\|_2<\frac{\varepsilon_n}{2} \frac{1}{C_{max}\sqrt{\widetilde{a}_n}}
$
implies
$\|f_0-f_{\wh{\mE},\mB}\|_n<\varepsilon_n.$  Therefore we have 
$$ 
\Pi\left(f\in\mF(\wh{\mE}): \|f - f_0\|_n\leq \varepsilon_n\right)>\Pi\left(\mB\in\R^{\widetilde a_n}:\|\mB-\wh{\mB}\|_2<\frac{\varepsilon_n}{2} \frac{1}{C_{max}\sqrt{\widetilde{a}_n}} \right).
$$
 Moreover, because  $\mu(\widetilde{\Omega}_k)\geq C^2_{min}/K(\wh{\mE})$ for some $C_{min}<1$, we have
$$
\|\wh{\mB}\|_2\leq \frac{\sqrt{a_n}}{C_{min}\lambda_{min}(\wh{\mE})}\|f_{\wh{\mE},\wh{\mB}}\|_n\leq \frac{\sqrt{a_n}}{C_{min}}\left(\frac{\varepsilon_n}{2}+\|f_0\|_\infty\right),
$$
where we used the fact $\lambda^2_{min}(\wh{\mE})\geq 1$ (proof of Lemma \ref{lemma:approx:ensemble}). Therefore we have $\|\wh{\mB}\|_2^2\leq C_2 a_n\|f_0\|_\infty^2$ for some $C_2>0$.
Following the calculations from Section \ref{sec:proof:tree1} (namely \eqref{aux2}),  we continue to lower-bound \eqref{lb}  with
\begin{equation}\label{eq:ratio2}
\frac{\pi(\wh{T})\pi(\wh{\bm{K}}\C \wh{T}) \pi(q_0)\e^{-\frac{\varepsilon_n^2}{8C^2_{max}\wt{a}_n}-{a}_n(C_2\|f_0\|_\infty^2 +\log2)}}{\left(\frac{\e\, p}{q_0}\right)^{q_0}[(\log_2a_n+1)\,q_0\,n]^{\widetilde{a}_n}} \left(\frac{\varepsilon_n^2}{4C^2_{max}\wt{a}_n}\right)^\frac{\widetilde{a}_n}{2}   \left(\frac{2}{\widetilde{a}_n}\right)^{\widetilde{a}_n/2 + 1}.
\end{equation}
This quantity should be at least $\e^{-d\,n\varepsilon_n^2}$ for some suitably large $d>0$.
Now,  with our prior \eqref{prior:KK} we can write
\begin{align*}
\pi(\wh{\bm{K}}\C \wh{T})&\gtrsim\prod_{t=1}^{\wh{T}}\e^{\wh{K}^t\log(\lambda/\wh{T}) -\wh{K}^t\log \wh{K}^t}
\gtrsim \e^{ - \widetilde a_n \log (\log_2{a_n} + 1) + \widetilde a_n \log(2\lambda/a_n)}.
\end{align*}
This quantity can be lower-bounded by $ \e^{ -  D\,\widetilde a_n \log a_n}
$ for some $D>0$. Then we can write 
\begin{align*}
&\frac{\pi(\wh{T})\pi(\wh{\bm{K}}\C \wh{T}) \pi(q_0)}{\left(\frac{\e\, p}{q_0}\right)^{q_0}[(\log_2a_n+1)\,q_0\,n]^{\widetilde{a}_n}}>\\
&\quad\quad\quad\quad\quad\quad{\e^{-C_T a_n/2}} \e^{-D\,\widetilde a_n \log{a_n} - q_0 \log(c\,\e\, p^{a+1}/q_0) } 
\e^{-\widetilde{a}_n \log(q_0 n (\log_2 {a_n} + 1)) }. 
\end{align*}
By our assumptions $q_0\lesssim\log^\beta n$ and $p\lesssim n^{q_0/(2\alpha+q_0)}$, the term $\e^{- q_0 \log(c\,\e\,p^{a+1}/q_0)}$ will safely be larger than $\e^{-d_1n\varepsilon_n^2}$ for some $d_1>0$. We take all the remaining important terms in \eqref{eq:ratio2}, aiming to show that $(i) \ \widetilde{a}_n \log(\widetilde{a}_n/\varepsilon_n^2)$, $(ii) \ a_n\|f_0\|_\infty^2$, $(iii) \ \widetilde a_n \log(q_0 n \log_2{a_n})$ and $(iv) \ \widetilde a_n \log{\widetilde a_n}$ are  bounded by a constant multiple of $n\varepsilon_n^2$. 
From \eqref{bound}, we obtain $a_n\lesssim n^{q_0/(2\alpha+q_0)}$ under our assumption $2C_0q_0\lesssim\log^\beta n$. Then we can write
\begin{equation}\label{atilde}
\widetilde a_n=a_n/2(\log_2a_n+1)\lesssim n^{q_0/(2\alpha+q_0)}\log n.
\end{equation}
Using this bound,  we verify that $(i)$-$(iv)$ are bounded by a constant multiple of $n\varepsilon_n^2 = n^\frac{q_0}{2\alpha + q_0}\log^{2\beta}n$.
First, note that
$$\widetilde{a}_n \log(\widetilde{a}_n/\varepsilon_n^2)\lesssim n^\frac{q_0}{2\alpha + q_0}\log n \log\left(n^\frac{q_0}{2\alpha + q_0}n^{\frac{2\alpha}{2\alpha + q_0}}\log^{1-2\beta}n  \right).$$ 
This quantity is bounded by a multiple of $ n^\frac{q_0}{2\alpha + q_0}\log^{2\beta}n$ when $\beta \geq 1$.
 Next, we can write $
 a_n\|f_0\|_\infty^2\lesssim n^\frac{q_0}{2\alpha + q_0} \log^{2\beta}n.$
Lastly,  it follows from \eqref{atilde} that $ \widetilde a_n \log{\widetilde a_n}\lesssim n^{q_0/(2\alpha+q_0)}\log^2 n$ and $ \widetilde a_n \log(q_0 n \log_2 a_n)\lesssim n^{q_0/(2\alpha+q_0)}\log^2 n$  . 
To sum up, there exists $d>0$ such that $\Pi(f\in\mF_\mE : \|f - f_0\|_n \leq \varepsilon_n)>\e^{-d\,n\varepsilon_n^2}$ for $\beta\geq1$.
\subsection{Condition \eqref{eq:remain}}\label{sec:last_one}
The condition  $\Pi(\mF_\mE\backslash\mF_\mE^n)\e^{(d+2)\, n\varepsilon_n^2}\rightarrow 0$  is verified similarly as in Section \ref{sec:remain:ens} and Section \ref{sec:proof:tree:eq:remain}. For $\Pi(q>q_n)$, we use the bound from Section \ref{sec:proof:tree:eq:remain} with {$q_n=\lceil C_q\mathrm{min}\{p,n^{q_0/(2\alpha+q_0)}\log^{2\beta} n/\log (p\vee n)\}\rceil$} and a large enough constant $C_q$. For $\Pi\left( (T,\bm{K}):\sum_{t=1}^TK^t>z_n\right)$ we use the bound from Section \ref{sec:remain:ens} with $z_n=\lfloor C_zn\varepsilon_n^2/\log n\rfloor$ and a large enough constant $C_z$.
 
\section{Proof of Lemma \ref{lemma:approx}}\label{proof:lemma:approx}
We start with an auxiliary statement showing that  when $f$ is $\alpha$-H\"{o}lder continuous, we can grow a  step function on any given (tree) partition so that the approximation error will be
governed by { cell diameters.

 Namely, for $f \in \Ha_{p}\,\cap\,\mC(\mS)$ and a valid tree partition $\mT=\{\Omega_k\}_{k=1}^K\in \mV_\mS^K$, there exists  a step function $f_{\mT,\,\wh{\b}}(\x) = \sum_{k=1}^{K} \wh{\beta}_k\,\1_{\Omega_k}(\x)$  such that 
$
\|f - f_{\mT,\,\wh{\b}}\|_n   \leq ||f||_{\Ha} \sqrt{\sum_{k=1}^K\mu(\Omega_k) \mathrm{diam}^{2\alpha}(\Omega_k;\mS)}.
$ 
Indeed, given $\mT$ and design points $\mX$, we take $ \wh{\beta}_k=\frac{1}{n_k}\sum_{i=1}^nf(\x_i)\1_{\Omega_k}(\x_i)$, where $n_k=\mu(\Omega_k)n$. Then for $\x_j\in\Omega_k\cap \mX$ we have, from H\"{o}lder continuity, 
\begin{align*}
|f(\x_j)-f_{\mT,\wh{\b}}(\x_j)|&<\frac{1}{n_k}\sum_{\x_i\in\Omega_k}|f(\x_j)-f(\x_i)|\\
& \leq ||f||_{\Ha} \frac{1}{n_k}\sum_{\x_i\in\Omega_k}||\x_{j\mS}-\x_{i\mS}||_2^\alpha\leq  ||f||_{\Ha}\, \mathrm{diam}^\alpha(\Omega_k;\mS).
\end{align*}
Then the approximation error satisfies
\begin{align}\label{eq:approx_gap}
||f- f_{\mT,\wh{\b}}||_n& \leq ||f||_{\Ha} \sqrt{\sum_{k=1}^K\mu(\Omega_k) \mathrm{diam}^{2\alpha}(\Omega_k;\mS)}. \,\,\qedhere
\end{align}

To continue with the proof of  \eqref{approx2}, we grow a $k$-$d$ tree partition $\wh{\mT}=\{\wh{\Omega}_k\}_{k=1}^K$ (as explained in Remark \ref{remark:kd}) and construct an approximating step function $f_{\wh{\mT},\wh{\b}}$, as outlined above.

Using  \eqref{eq:approx_gap}, the statement \eqref{approx2} then follows from
\begin{align*}
||f- f_{\wh{\mT},\wh{\b}}||_n
&<C_{\max}\,||f||_{\Ha}\,\max\limits_{1\leq k\leq K}\mathrm{diam}^\alpha(\wh{\Omega}_k;\mS),
\end{align*}
where we used the fact that $\mu(\wh{\Omega}_k)\leq C^2_{max}/K$  in $k$-$d$ trees.
A minor modification of the proof of  Proposition 6 in \cite{verma2009} yields
$\sum_{k=1}^K\mu(\wh{\Omega}_k)\mathrm{diam}(\wh{\Omega}_k;\mS)\leq  \frac{{q}}{K^{1/q}}.$ The rest follows from Definition \ref{def:regular}. \hspace{5cm} $\square$
}

\section{Auxiliary Result}
\begin{lemma} \label{lem:eigenvalupper}
Assume a valid ensemble $\mE$ consisting of $T$ trees, each with $K^t$ leaves.  Let $\lambda_{max}^2(\mE)$  be the largest eigenvalue of $\wt \A(\mE)=\A(\mE)'\A(\mE)$. Then
\begin{equation}\label{continue2}
\lambda_{max}^2(\mE)\leq  K(\mE) (T\times \bar{K}),
\end{equation}
where $\bar{K}=\frac{1}{T}\sum_{t=1}^TK^t$ and where $K(\mE)$ denotes the number of rows of $\A(\mE)$.
\end{lemma}

\begin{proof}
 By the Gershgorin  circle theorem, all eigenvalues of  $\wt{\A}(\mE)=\left(\widetilde{a}_{ij}\right)$ lie inside the union of intervals $[\wt{a}_{ii} - \sum_{j \neq i} \wt{a}_{ij}, \wt{a}_{ii} + \sum_{j \neq i} \wt{a}_{ij}]$ for $i = 1, \ldots, T\times\bar{K}$. As explained in Section \ref{sec:bart}, the diagonal and off-diagonal entries of $\wt{\A}(\mE)$  quantify the persistence and the overlap in terms of the number of intersecting global partitioning cells. The  magnitude $|\widetilde{a}_{ij}|$  is no larger than $K(\mE)$ for each $1\leq i,j\leq T\times\bar{K}$. The upper bound on the maximal eigenvalue $\lambda^2_{max}(\mE)$ is thus $K(\mE)(T\times \bar{K})$. 
  \qedhere
\end{proof}

\end{document}